\documentclass{amsart}
\usepackage{amsmath,amssymb,graphicx,amsthm,bm,hyperref,mathrsfs}
\usepackage{graphicx}
\usepackage{hyperref}
\hypersetup{colorlinks=true,citecolor=blue,linkcolor=cyan}
\newtheorem{thm}{Theorem}[section]
\newtheorem{cor}[thm]{Corollary}
\newtheorem{lem}[thm]{Lemma}
\newtheorem{prop}[thm]{Proposition}
\theoremstyle{definition}
\newtheorem{defn}[thm]{Definition}
\theoremstyle{remark}
\newtheorem{rem}[thm]{Remark}
\numberwithin{equation}{section}

\newcommand{\F}{\mathcal{F}}
\newcommand{\Ff}{\mathbb{F}}
\newcommand{\R}{\mathcal{R}}
\newcommand{\s}{\mathcal{S}}
\newcommand{\Gal}{\mbox{Gal}}
\newcommand{\Z}{\mathbb{Z}}
\newcommand{\Hh}{\mathcal{H}}
\newcommand{\Tr}{\mbox{Tr}}
\newcommand{\Frob}{\mbox{Frob}}
\newcommand{\Prob}{\mbox{Prob}}
\newcommand{\Pp}{\mathbb{P}}

\newcommand{\Mod}[1]{\ (\text{mod}\ #1)}
\newcommand{\ord}{\mbox{ord}}
\newcommand{\lcm}{\mbox{lcm}}
\begin{document}

\title[Distribution of Points on Abelian Covers Over Finite Fields]{Distribution of Points on Abelian Covers Over Finite Fields}%
\author{Patrick Meisner}%


\begin{abstract}

We determine in this paper the distribution of the number of points on the covers of $\Pp^1(\Ff_q)$ such that $K(C)$ is a Galois extension and $\Gal(K(C)/K)$ is abelian when $q$ is fixed and the genus, $g$, tends to infinity. This generalizes the work of Kurlberg and Rudnick and Bucur, David, Feigon and Lalin who considered different families of curves over $\Ff_q$. In all cases, the distribution is given by a sum of $q+1$ random variables.

\end{abstract}
\maketitle


\section{Introduction}\label{intro}

Let $q$ be a power of a prime and $C$ a smooth, projective curve over $\Ff_q$. Denote $\Ff_q(X)$ as $K$ and $K(C)$ as the field of functions of $C$. Then $K(C)$ is a finite extension of $K$. Moreover, if we fix a copy of $\Pp^1(\Ff_q)$, then every finite extension of $K$ corresponds to smooth, projective curve (Corollary 6.6 and Theorem 6.9 from Chapter I of \cite{hart}).

If $K(C)/K$ is Galois, then denote $\Gal(C)= \Gal(K(C)/K)$. Let $g(C)$ be the genus of $C$. Define the family of smooth, projective curves
$$\Hh_{G,g} = \{C : \Gal(C)=G, g(C)=G\}.$$
We want to determine the probability, that a random curve in this family has a given number of points. That is, for every $N\in \Z_{\geq 0}$, we want to determine
$$\Prob(C\in\Hh_{G,g} : \#C(\Pp^1(\Ff_q)) = N) = \frac{|\{C\in \Hh_{G,g} :  \#C(\Pp^1(\Ff_q)) = N\}|}{|\Hh_{G,g}|}$$

It is well know that
$$\#C(\Pp^1(\Ff_q)) = q+1 - \Tr(\Frob_q)$$
where $\Frob_q$ is the $q^{th}$-power Frobenius. Moreover, a classical result due to Katz and Sarnak \cite{KS} says that if $g$ is fixed and we let $q$ tend to infinity then the trace of the Frobenius in a family is distributed like the trace of a random matrix in the monodromy group associated to the family. We will be interested in what happens when $q$ is fixed and we let $g$ tend to infinity.

Several cases of this are known for specific family of groups. It was first done by Kurblerg and Rudnick \cite{KR} for hyper-elliptic curves ($G=\Z/2\Z$). This was extended by Bucur, David, Feigon and Lalin \cite{BDFL1},\cite{BDFL2} for prime cyclic curves ($G=\Z/p\Z$, $p$ a prime). Lorenzo, Meleleo and Milione \cite{LMM} then determined this for $n$-quadratic curves ($G=(\Z/2\Z)^n$). More recently the author \cite{M1} extended the work of Bucur, David, Feigon and Lalin to the case of arbitrary cyclic curves ($G=\Z/r\Z$, $r$ not necessarily a prime).

In all the works mentioned above the probability is not determine for the whole family $\Hh_{G,g}$ but instead for an irreducible moduli space of the family. That is, we can write
$$\Hh_{G,g} = \bigcup_{\vec{d}(\vec{\alpha})} \Hh^{\vec{d}(\vec{\alpha})}$$
where $\vec{d}(\vec{\alpha}) = (d(\vec{\alpha}))_{\vec{\alpha}}$ in a non-negative integer valued vector indexed by a set of $|G|-1$ vectors (the $\vec{\alpha}$) and the union is over all such vectors that satisfy a linear equation and a set of linear congruence conditions. Moreover, $\Hh^{\vec{d}(\vec{\alpha})}$ is a set of tuples of polynomials of prescribed degree that correspond to a curve with Galois group $G$ and genus $g(C)$. See Section \ref{genformsec} for a full description of these sets.

\begin{rem}\label{conddisc}

There is a natural correspondence between the genus of the curve and the degree of the discriminant of $K(C)$. Through this correspondence we can view $\Hh^{\vec{d}(\vec{\alpha})}$ as the set of curves such that the degree of the conductor of $K(C)$ is fixed. Then the linear relationships that the union is over is the conductor-discriminant formula.

\end{rem}

Moreover, all the previous results restrict to the case that $q\equiv 1 \mod{\exp(G)}$ where $\exp(G) = \min(n : ng=e \mbox{ for all } g\in G)$. This is in order to use Kummer theory to get a classification of the curves. Therefore, our main result will be for this irreducible moduli space under this same assumption.

\begin{thm}\label{mainthm2}
Let $G=\Z/r_1\Z\times\dots\times\Z/r_n\Z$ such that $r_j|r_{j+1}$ and fix $q$ such that $q\equiv 1 \Mod{r_n}$ then as $d(\vec{\alpha})\to\infty$ for all $\vec{\alpha}\in\R$,

$$\frac{|\{C\in\Hh^{\vec{d}(\vec{\alpha})} : \#C(\Pp^1(\Ff_q)) = M\}|}{|\Hh^{\vec{d}(\vec{\alpha})}|} \sim\Prob\left(\sum_{i=1}^{q+1} X_i= M\right)$$
where the $X_i$ are $i.i.d.$ random variables taking value $0$ or $\frac{|G|}{s}$ for some $s|r_n$ such that
$$X_i = \begin{cases} \frac{|G|}{s} & \mbox{with probability }  \frac{s\phi_G(s)} {|G|(q+|G| -1)} \mbox{ if } s\not=1 \\ |G| & \mbox{with probability }  \frac{q}{|G|(q+|G|-1)} \\ 0 & \mbox{with probability } \frac{(|G|-1)(q+|G|) -\sum_{s|r_n}s\phi_G(s)+1}{|G|(q+|G|-1)}    \end{cases}$$
where $\phi_G(s)$ is the number of elements of $G$ of order $s$.
\end{thm}

\begin{rem}

Notice that in our result, we require $d(\vec{\alpha})$ to tend to infinity for all $\vec{\alpha}$. This implies that the genus tends to infinity as the genus can be written as a linear combination of the $d(\vec{\alpha})$. However, the converse is not true. That is, if $g$ tends to infinity, this only implies that at least one of the $d(\vec{\alpha})$ would tend to infinity. In this case, the error term would not necessarily go to zero. Bucur, David, Feigon, Kaplan, Lalin, Ozman and Wood \cite{BDFK+} solve this problem for the whole space $\Hh_{G,g}$ where $G$ is a prime cyclic. Work is done towards extending this by the author to any abelian group in a forthcoming paper  with success in the case $G$ is a power of a prime cyclic ($G=(\Z/p\Z)^n$, $p$ a prime).


\end{rem}


\section{Genus Formula and Irreducible Moduli Space}\label{genformsec}

In this section we will first determine a formula for the genus of the curve and from this formula create the irreducible moduli spaces $\Hh^{\vec{d}(\vec{\alpha})}$.

Let $C$ be a curve such that $\Gal(C)=G$ is abelian. Then we can find unique $r_j$ such that $r_j|r_{j+1}$ and $G=\Z/r_1\Z\times\dots\times\Z/r_n\Z$. Therefore, $\exp(G)=r_n$. Since we are assuming $q\equiv1\Mod{r_n}$, we get that $\mu_{r_n}\subset K$ and hence $K(C)/K$ is a Kummer extension. Then Kummer Theory (Chap.14 Proposition 37 of \cite{DF}) tells us that there exists $F_1, \dots, F_n\in \Ff_q[X]$ such that $F_j$ is $r_j^{th}$-power free and
$$K(C) = K(\sqrt[r_1]{F_1},\dots,\sqrt[r_n]{F_n}).$$

Let $g=g(C)$, be the genus of the curve $C$. Then the Riemann-Hurwitz formula (Theorem 7.16 of \cite{rose}), says that
\begin{align}\label{genform4}
2g+2|G|-2 = \sum_{\mathfrak{P}}(e(\mathfrak{P}/P)-1)\deg_{K(C)}(\mathfrak{P})
\end{align}
where the sum is over all primes $\mathfrak{P}$ of $K(C)$, $e(\mathfrak{P}/P)$ is the ramification index and $\deg_{K(C)}(\mathfrak{P})$ is the dimension of $\mathcal{O}_{K(C)}/\mathfrak{P}$ as a vector space over $\Ff_q$. By Proposition 7.7 of \cite{rose}, we get that if $\mathfrak{P}|P$, then $\deg_{K(C)}(\mathfrak{P}) = f(\mathfrak{P}/P)\deg_K(P)$, where $f(\mathfrak{P}/P)$ is the inertia degree and $\deg_K(P)$ is the degree of the polynomial $P$.  Moreover, since our extension is Galois, we get that for any $\mathfrak{P}_1,\mathfrak{P}_2|P$, $e(\mathfrak{P}_1/P)=e(\mathfrak{P}_2/P):=e(P)$ and $f(\mathfrak{P}_1/P)=f(\mathfrak{P}_2/P):=f(P)$. Hence,
$$\sum_{\mathfrak{P}|P}(e(\mathfrak{P}/P)-1)\deg_{K(C)}(\mathfrak{P}) = g(P)(e(P)-1)f(P)\deg_K(P)$$
$$ = \left(|G|-\frac{|G|}{e(P)}\right)\deg_K(P),$$
where $g(P)$ is the number of $\mathfrak{P}|P$.

Therefore, \eqref{genform4} becomes
\begin{align}\label{genform5}
2g+2|G|-2 = \sum_{P}\left(|G|-\frac{|G|}{e(P)}\right)\deg_K(P)
\end{align}
where the sum is over all the primes in $K$. Hence it is enough to determine the ramification index for all $P$ in $K$.

\begin{lem}\label{ramlem1}

Let $K \subset L \subset L(\sqrt[r]{F(X)})=L'$ be an extension of fields where $F\in \Ff_q[X]$ is $r^{th}$-power free and $[L':L]=r$. Let $\mathfrak{P}$ be a prime in $L$ and $\mathfrak{P}'$ be a prime in $L'$, lying over $\mathfrak{P}$. If $\ord_{\mathfrak{P}}(F) = n$, then $e(\mathfrak{P}'/\mathfrak{P}) = \frac{r}{(r,n)}$.

\end{lem}

\begin{proof}

Since $[L':L]=r$, the characteristic polynomial is $Y^r-F(X)$. We can write $F(X)=F_1(X)F_2(X)^n$ where $\ord_{\mathfrak{P}}(F_2(X))=1$ and $(F_1(X)\mathcal{O}_{L},\mathfrak{P})=1$. Then $Y^r - F(X) \equiv Y^r \Mod{\mathfrak{P}}$. Hence,
$$\mathfrak{P}' = \mathfrak{P}\mathcal{O}_{L'} + \sqrt[r]{F(X)}\mathcal{O}_{L'}$$
will be a prime lying over $\mathfrak{P}$.

Now, $e(\mathfrak{P}'/\mathfrak{P})$ will be the smallest integer $e$ such that $(\mathfrak{P}')^e \subset \mathfrak{P}\mathcal{O}_{L'}$. We have that
$$(\mathfrak{P}')^e =  \sum_{j=0}^{e}\mathfrak{P}^{e-j}\left(\sqrt[r]{F(X)}\mathcal{O}_{L'}\right)^j.$$
Now,
$$\sum_{j=0}^{e-1}\mathfrak{P}^{e-j}\left(\sqrt[r]{F(X)}\mathcal{O}_{L'}\right)^j \subset \mathfrak{P}O_{L'}$$
so it remains to determine when $\left(\sqrt[r]{F(X)}\mathcal{O}_{L'}\right)^e\subset \mathfrak{P}O_{L'}$. Finally,
$$\left(\sqrt[r]{F(X)}\mathcal{O}_{l}\right)^e = \left(\sqrt[r]{F_1(X)F_2(X)^n}\mathcal{O}_{L'}\right)^e = \left(\sqrt[\frac{r}{(r,n)}]{F_2(X)^{\frac{n}{(r,n)}}}\sqrt[r]{F_1(X)} \mathcal{O}_{L'}\right)^e$$
and we see that $e(\mathfrak{P}'/\mathfrak{P})=\frac{r}{(r,n)}$.

\end{proof}

\begin{lem}\label{ramlem2}

Let $K \subset L \subset L(\sqrt[r_1]{F_1(X)})=L'\subset L(\sqrt[r_1]{F_1(X)}, \sqrt[r_2]{F_2(X)}) = L''$ be extensions of fields where $F_1,F_2\in\Ff_q[X]$ are $r_1^{th}$ and $r_2^{th}$-power free respectively and $[L':L]=r_1, [L'':L']=r_2$. Let $\mathfrak{P}$ be a prime in $L$ and $\mathfrak{P}''$ be a prime in $L''$ lying above $\mathfrak{P}$. If $\ord_{\mathfrak{P}}(F_1) = n$ and $\ord_{\mathfrak{P}}(F_2)=m$, then $e(\mathfrak{P}''/\mathfrak{P})=\lcm\left(\frac{r_1}{(r_1,n)},\frac{r_2}{(r_2,m)}\right)$

\end{lem}

\begin{proof}

Let $\mathfrak{P}'$ be a prime in $L'$ such that $\mathfrak{P}''|\mathfrak{P}'|\mathfrak{P}$, then by Lemma \ref{ramlem1}, $e(\mathfrak{P}'/\mathfrak{P})=\frac{r_1}{(r_1,n)}$. Therefore, $\ord_{\mathfrak{P}'}(F_2)=m\frac{r_1}{(r_1,n)}$ and, again by Lemma \ref{ramlem1}, $e(\mathfrak{P}''/\mathfrak{P}')=\frac{r_2}{(r_2,m\frac{r_1}{(r_1,n)})}$. Hence, $e(\mathfrak{P}''/\mathfrak{P}) = \frac{r_1}{(r_1,n)}\frac{r_2}{(r_2,m\frac{r_1}{(r_1,n)})}$. So it remains to show that this is $\lcm\left(\frac{r_1}{(r_1,n)},\frac{r_2}{(r_2,m)}\right)$.

Let $A,B,C$ be positive integers. We will show that $A\frac{B}{(B,AC)} = \lcm(A,\frac{B}{(B,C)})$. Let $A = \prod p_i^{a_i}$, $B=\prod p_i^{b_i}$, $C= \prod p_i^{c_i}$. Then the left hand and right hand sides are
$$\prod p_i^{a_i + b_i - \min(b_i,a_i+c_i)} \quad \quad \quad \prod p_i^{\max(a_i, b_i - \min(b_i,c_i))}$$
respectively. If $b_i\leq a_i+c_i$, then the left hand exponent becomes $a_i$. Moreover, $b_i\leq c_i$ so the right hand exponent would become $\max(a_i,b_i-c_i)=a_i$ as $a_i \geq b_i-c_i$. If $b_i \geq a_i+c_i$ then the left hand exponent becomes $b_i-c_i$. Further, $b_i\geq c_i$ so then the right hand exponent would become $\max(a_i, b_i-c_i) = b_i-c_i$ as $a_i \leq b_i-c_i$. This completes the proof.

\end{proof}

So, we see that in order to determine the genus, we need to keep track of $\ord_P(F_j)$ for all $P\in\Ff_q[X]$ and  $1\leq j \leq n$. Towards this define the set
$$\R = [0,\dots,r_1-1]\times\dots\times[0,\dots,r_n-1] \setminus\{(0,\dots,0)\}$$
to be the set of integer-valued vectors with $j$ entry between $0$ and $r_j-1$ such that not all entries are $0$. Write an element of $\R$ as $\vec{\alpha}=(\alpha_1,\dots,\alpha_n)$. Then, for every $\vec{\alpha}\in\R$, let
$$f_{\vec{\alpha}} = \prod_{\substack{P\\\ord_P(F_j)=\alpha_j}} P$$
where the product is over all (finite) monic prime polynomials of $\Ff_q[X]$. Then we can write
$$F_j = c_j\prod_{\vec{\alpha}\in\R} f_{\vec{\alpha}}^{\alpha_j}$$
for some $c_j\in \Ff_q^*$ where we use the convention that $f^0$ is identically the constant polynomial $1$.

\begin{prop}\label{ramprop}

If $P|f_{\vec{\alpha}}$ then $e(P)=\underset{j=1,\dots,n}{\lcm}\left(\frac{r_j}{(r_j,\alpha_j)}\right)$

\end{prop}

\begin{proof}

If $P|f_{\vec{\alpha}}$ then $\ord_P(F_j)=\alpha_j$ for all $j$. Thus if we recursively apply Lemma \ref{ramlem2}, we get the result.

\end{proof}

If $P_{\infty}$ is the prime at infinity, then we see that $\ord_{P_{\infty}}(F) = \deg(F)$. Therefore, if $\deg(F_j)=d_j$,
$$e(P_{\infty}) = \underset{j=1,\dots,n}{\lcm}\left(\frac{r_j}{(r_j,d_j)}\right).$$

Therefore, we can rewrite \eqref{genform5} as
\begin{align}\label{genform6}
2g+2|G|-2 =  \sum_{\vec{\alpha}\in\R} \left(|G|-\frac{|G|}{e(\vec{\alpha})}\right)\deg(f_{\vec{\alpha}}) + |G|-\frac{|G|} {e(\vec{d})}
\end{align}
where $\vec{d} = (d_1,\dots,d_n)$ and for any $\vec{v}=(v_1,\dots,v_n)$,
$$e(\vec{v}) = \underset{j=1,\dots,n}{\lcm}\left(\frac{r_j}{(r_j,v_j)}\right). $$

Thus, what we want to keep track of is $\deg(f_{\vec{\alpha}})$. Hence, we will let $d(\vec{\alpha})$ be a non-negative integer for all $\vec{\alpha}\in\R$ and
$$\vec{d}(\vec{\alpha}) = (d(\vec{\alpha})_{\vec{\alpha}\in\R}$$
be a vector indexed by the elements of $\vec{\alpha}\in\R$. Moreover, for every $\vec{d}(\vec{\alpha})$ define
$$d_j :=\sum_{\vec{\alpha}\in\R} \alpha_j d(\vec{\alpha})$$
for $j=1,\dots,n$.

Define the sets
$$\F_d = \{f : f, \mbox{ monic, squarefree and } \deg(f)=d\}$$
$$\F_{\vec{d}(\vec{\alpha})} = \{(f_{\vec{\alpha}})\in \prod_{\vec{\alpha}\in\R} \F_{d(\vec{\alpha})}: (f_{\vec{\alpha}},f_{\vec{\beta}})=1 \mbox{ for all } \vec{\alpha}\not=\vec{\beta} \}.$$
That is, the set of monic, square-free and pairwise coprime tuples of polynomials with prescribed degrees.

Consider $\vec{d}(\vec{\alpha})$ such that $d_j \equiv 0 \mod{r_j}$ for $j=1,\dots,n$, then for $(f_{\vec{\alpha}})\in\F_{\vec{d}(\vec{\alpha})}$,  the right side of \eqref{genform6} becomes
$$\sum_{\vec{\alpha}\in\R} \left(|G|-\frac{|G|}{e(\vec{\alpha})}\right)d(\vec{\alpha}).$$

Now, consider $\vec{d}(\vec{\alpha})$ such that $d_j \equiv r_j-\beta_j \mod{r_j}$ for some $\vec{\beta}\in\R$. Define $d'(\vec{\beta})=d(\vec{\beta})+1$ and $d'(\vec{\alpha})=d(\vec{\alpha})$ for $\vec{\alpha}\not=\vec{\beta}$ and $\vec{d}'(\vec{\alpha}) = (d'(\vec{\alpha}))$. Then
$$d'_j :=\sum_{\vec{\alpha}\in\R}\alpha_jd'(\vec{\alpha}) \equiv 0 \mod{r_j}.$$

This motivates define the set
$$\F^{\vec{\beta}}_{\vec{d}(\vec{\alpha})} = \{(f_{\vec{\beta}},(f_{\vec{\alpha}}))\in \F_{d(\vec{\beta})-1} \times \prod_{\substack{\vec{\alpha}\in\R \\ \vec{\alpha}\not=\vec{\beta}}} \F_{d(\vec{\alpha})}: (f_{\vec{\alpha}},f_{\vec{\gamma}})=1 \mbox{ for all } \vec{\alpha},\vec{\gamma}\in\R   \}.$$
This set is the same as the previous set except that the degree is dropped by $1$ in the $\vec{\beta}^{th}$-coordinate.

Therefore, by the above argument we get that any tuple $(f_{\vec{\alpha}})$ lives in a unique $\F^{\vec{\beta}}_{\vec{d}(\vec{\alpha})}$ such that $d_j\equiv 0 \mod{r_j}$.

Hence if we define the set
$$\F_{[\vec{d}(\vec{\alpha})]} = \F_{\vec{d}(\vec{\alpha})} \cup \bigcup_{\vec{\beta}\in\R} \F^{\vec{\beta}}_{\vec{d}(\vec{\alpha})}$$
then as $\vec{d}(\vec{\alpha})$ runs over all vectors such that $d_j\equiv 0 \mod{r_j}$, we get that the set $\F_{[\vec{d}(\vec{\alpha})]}$ runs over all tuples. Therefore, from now on we will always be assuming $d_j\equiv 0 \mod{r_j}$, $j=1,\dots,n$.

Moreover, the genus of the curves corresponding to the tuples in $\F_{[\vec{d}(\vec{\alpha})]}$ is invariant. Indeed, if $(f_{\vec{\alpha}})\in\F_{\vec{d}(\vec{\alpha})}$, then we get that the genus, $g$, satisfies
$$2g+2|G|-2 = \sum_{\vec{\alpha}\in\R} \left(|G|-\frac{|G|}{e(\vec{\alpha})}\right)d(\vec{\alpha}).$$
Further, if $(f_{\vec{\alpha}})\in\F^{\vec{\beta}}_{\vec{d}(\vec{\alpha})}$ the genus, $g'$, satisfies
\begin{align*}
2g'+2|G|-2 & = \sum_{\substack{\vec{\alpha}\in\R \\ \vec{\alpha}\not=\vec{\beta}} } \left(|G|-\frac{|G|}{e(\vec{\alpha})}\right)d(\vec{\alpha}) + \left(|G|-\frac{|G|}{e(\vec{\alpha})}\right)(d(\vec{\beta})-1) + |G|-\frac{|G|}{e(\vec{d})} \\
& = \sum_{\vec{\alpha}\in\R} \left(|G|-\frac{|G|}{e(\vec{\alpha})}\right)d(\vec{\alpha})\\
& = 2g+2|G|-2.
\end{align*}

Now, we need to add information about the leading coefficients, so define
$$\hat{\F}_{\vec{d}(\vec{\alpha})} = (\Ff_q^*)^n\times\F_{\vec{d}(\vec{\alpha})}$$
$$\hat{\F}^{\vec{\beta}}_{\vec{d}(\vec{\alpha})} =(\Ff_q^*)^n\times \F^{\vec{\beta}}_{\vec{d}(\vec{\alpha})}$$
$$\hat{\F}_{[\vec{d}(\vec{\alpha})]} = (\Ff_q^*)^n\times\F_{[\vec{d}(\vec{\alpha})]}$$

Every element of $\hat{\F}_{[\vec{d}(\vec{\alpha})]}$ corresponds to a curve and every curve corresponds to an element of $\hat{\F}_{[\vec{d}(\vec{\alpha})]}$. With that being said, we define
$$\Hh^{\vec{d}(\vec{\alpha})} = \{C : C \mbox{ corresponds to an element of } \hat{\F}_{[\vec{d}(\vec{\alpha})]} \}$$
and we get
$$\Hh_{G,g} = \bigcup \Hh^{\vec{d}(\vec{\alpha})}$$
where the union is over all $\vec{d}(\vec{\alpha})$ that satisfy
$$2g+2|G|-2 = \sum_{\vec{\alpha}\in\R} \left(|G|-\frac{|G|}{e(\vec{\alpha})}\right) d(\vec{\alpha})$$
$$\sum_{\vec{\alpha}\in\R} \alpha_j d(\vec{\alpha}) \equiv 0 \mod{r_j}, j=1,\dots,n.$$


\section{Number of Points on the Curve}\label{abelnumpts}

In this section, we will find a formula for the number of points on a curve in $\Hh^{\vec{d}(\vec{\alpha})}$. To begin, we will determine a formula for the number of points lying above $x$ for all $x\in\Pp^1(\Ff_q)$. In order to do this, however, we need a \textit{smooth}, affine model of our curve at $x$.

We can view $K(C)$ as a vector space over $K$ with dimension $|G|$.  Let
$$\mathscr{B} = \{B_1,\dots,B_{|G|}\}$$
be a basis of $K(C)$ over $K$. Since $q\equiv 1 \Mod{\exp(G)}$, by Kummer Theory, we can assume that for all $B_i\in\mathscr{B}$, there exists an $m_i\in\Z_{>0}$ and $H_i\in \Ff_g[X]$ such that $H_i$ is $m_i^{th}$-power free and $B_i = \sqrt[m_i]{H_i}$. Now, if $x\in\Pp^1(\Ff_q)$, then we can find $H_{j_1}, \dots, H_{j_n} $ such that the smooth affine model of $C$ at $x$ is of the form
$$Y_1^{m_{j_1}} =H_{j_1}(X) \quad \quad Y_2^{m_{j_2}} = H_{j_2}(X) \quad \quad \dots \quad \quad Y_n^{m_{j_n}} = H_{j_n}(X)$$

Since $x$ is smooth in this model, at most one of the $H_{j_k}$ may have a root at $x$ of order at most $1$. Therefore, we see that the number of points lying over $x$ will be
$$\begin{cases} m_{j_1}m_{j_2}\cdots m_{j_n} & H_{j_i}(x)\in (\Ff_q^*)^{m_{j_i}}, i=1,\dots,n  \\ \frac{m_{j_1}m_{j_2}\cdots m_{j_n}}{m_{j_k}} & H_{j_k}(x)=0,  H_{j_i}(x)\in (\Ff_q^*)^{m_{j_i}}, i=1,\dots,n, i\not=k \\ 0 & \mbox{otherwise}\end{cases}$$

If we let $\chi_m:\Ff_q^* \to \mu_m$ be a multiplicative character of order $m$, and extend it to all of $\Ff_q$ by setting $\chi_m(0)=0$ then we see that we can write the number of points lying over $x$ as
$$\prod_{k=1}^n\left(1+ \sum_{i=1}^{m_{j_k}-1} \chi^i_{m_{j_k}}\left(H_{j_k}(x)\right)\right).$$

Let $B_i \not \in K(B_{j_1},\dots,B_{j_n})$. Then I claim that $H_i(x)=0$. Indeed, consider the smooth projective curve $C'$ such that $K(C')= K(B_{j_1},\dots,B_{j_n},B_i)$. Then $C'$ will have an affine model of the form
$$ Y_s^{m_i} = H_i(X)$$
$$ Y_k^{m_{j_k}} = H_{j_k}(X),  1\leq k \leq n, k\not=s.$$
That is $H_i$ will replace $H_{j_s}$ for some $1\leq s \leq n$.

Moreover, this affine model is not smooth at $x$ by our choices of $H_{j_1},\dots,H_{j_n}$. Therefore, one of four things may happen:

\begin{enumerate}
\item $H_{j_k}(X)$ is divisible by $(X-x)^2$ for some $1 \leq k \leq n$, $k\not=s$
\item $H_i(X)$ is divisible by $(X-x)^2$
\item $H_{j_k}(x)=H_{j_k'}(x)=0$ for some $1\leq k < k' \leq n$, $k,k'\not=s$.
\item $H_{j_k}(x) = H_i(x)=0$ for some $1\leq k \leq n$, $k\not=s$.
\end{enumerate}
Case one and three can't happen because this would imply our original model was not smooth at $x$. Therefore, case two or four must happen and in both of these cases $H_i(x)=0$

Hence, the number of points lying over $x$ is
$$\prod_{k=1}^n \left(1+ \sum_{i=1}^{m_{j_k}-1} \chi^i_{m_{j_k}}\left(H_{j_k}(x)\right)\right) = \sum_{j=1}^{|G|} \chi_{m_j}(H_j^{m_j}(x))$$
as all the terms appearing on the right hand side that don't appear on the left hand side are $0$.

Let $C\in\Hh^{\vec{d}(\vec{\alpha})}$ such that $C$ corresponds to $(\vec{c},(f_{\vec{\alpha}}))\in\hat{\F}_{[\vec{d} (\vec{\alpha})]}$. To use the discussion above, we want to find a basis for $K(C)$ over $K$ such that each element in the basis is an $m^{th}$ root of an $m^{th}$-powerfree polynomial. Towards this, define
$$\s = \{ \vec{s} = (s_1,\dots,s_n) : s_j|r_j\},$$
the set of vectors whose $j^{th}$ component divides $r_j$. For all $\vec{s}\in\s$ define
$$\ell(\vec{s}) = \lcm(s_1,\dots,s_n)$$
$$\Omega_{\vec{s}} = \{ \vec{\omega} = (\omega_1,\dots,\omega_n) : 1 \leq \omega_j \leq s_j, (\omega_j,s_j)=1\} \subset \R.$$
For any $\vec{s}\in\s$, $\vec{\omega}\in\Omega_{\vec{s}}$, and $(\vec{c},(f_{\vec{\alpha}}))\in\hat{\F}_{\vec{d}(\vec{\alpha})}$ define
$$F_{(\vec{s})}^{(\vec{\omega})}(X) := c_{(\vec{s})}^{(\vec{\omega})} \prod_{\vec{\alpha}\in\R} f_{\vec{\alpha}}(X)^{ \sum_{j=1}^n \frac{\ell(\vec{s})}{s_j} \omega_j\alpha_j \Mod{\ell(\vec{s})} }$$
$$c_{(\vec{s})}^{(\vec{\omega})} := \prod_{j=1}^n c_j^{\frac{\ell(s)}{s_j}\omega_j \Mod{\ell(\vec{s})}}.$$
When we write in the exponent $* \Mod{\ell(\vec{s})}$, we mean the smallest, non-negative integer that is congruent to $*$ modulo $\ell(\vec{s})$. Moreover, we make the identification that $f_{\vec{\alpha}}(X)^{0}$ is identically the constant polynomial $1$. Hence, if $\sum_{j=1}^n \frac{\ell(\vec{s})}{s_j} \omega_j\alpha_j \equiv 0 \Mod{\ell(\vec{s})} $, then $f_{\vec{\alpha}}(X)$ does not divide $F_{(\vec{s})}^{(\vec{\omega})}(X)$. In particular, if $\vec{s}=(1,\dots,1)$, then $\Omega_{\vec{s}} = \{(1,\dots,1)\}$ and we make the identification
$$F_{(1,\dots,1)}^{(1,\dots,1)}(X) =1, c_{(1,\dots,1)}^{(1,\dots,1)} =1$$

Therefore, we see that a basis for $K(C)$ over $K$ can be given by
$$\mathscr{B} = \left\{ \left(F_{(\vec{s})}^{(\vec{\omega})}(X)\right)^{\frac{1} {\ell(\vec{s})}}, \vec{s}\in\s, \vec{\omega}\in\Omega_{\vec{s}}   \right\}$$
This basis has the required property and hence the number of points lying over any $x\in\Ff_q$ can be written as
$$\sum_{\vec{s}\in\s} \sum_{\vec{\omega}\in\Omega_{\vec{s}}} \chi_{\ell(\vec{s})}\left(F_{(\vec{s})}^{(\vec{\omega})}(x) \right).$$

This leads to following lemma.

\begin{lem}\label{abelaffnumpts}

Let $C\in\Hh^{\vec{d}(\vec{\alpha})}$ that corresponds to $(\vec{c},(f_{\vec{\alpha}}))\in\hat{\F}_{[\vec{d}(\vec{\alpha})]}$. Then the number of affine points on the curve is
$$\#C(\Ff_q) = \sum_{x\in\Ff_q} \sum_{\vec{s}\in\s} \sum_{\vec{\omega}\in\Omega_{\vec{s}}} \chi_{\ell(\vec{s})}\left( F_{(\vec{s})}^{(\vec{\omega})}(x) \right).$$

\end{lem}

It remains to determine what happens at the point at infinity, $x_{q+1}$. For any $F(X)\in\Ff_q[X]$, let $\tilde{F}(X)$ denote the polynomial that inverts the order of the coefficients of $F(X)$. That is, if
$$F(X)= a_0+a_1X+\dots+a_dX^d,$$
then
$$\tilde{F}(X) = a_0X^d+a_1X^{d-1}+\dots+a_d.$$
Further, if we let $X' = 1/X$, then we have $F(X) = (X')^{-d}\tilde{F}(X')$, where $d=\deg(F)$. Hence to determine what happens at $x_{q+1}$, we need to determine what happens when $X'=0$ for the curve
$$Y_j^{r_j} = (X')^{-d_j} \tilde{F}_j(X'), j=1,\dots,n.$$
If we write $d_j = r_jm_j+k_j$ with $1\leq k_j \leq r_j$, and let $Y'_j = Y_j(X')^{m_j+1}$, then we have an isomorphism to the curve
$$(Y'_j)^{r_j} = (X')^{r_j-k_j} \tilde{F}_j(X'), j=1,\dots,n.$$
So, we see we get a root at $x_{q+1}$ if and only if $k_j\not=r_j$ if and only if $d_j\not\equiv 0 \Mod{r_j}$. Therefore, we can write
$$F_j(x_{q+1}) = \begin{cases} c_j & d_j \equiv 0 \Mod{r_j} \\ 0 & d_j \not \equiv 0 \Mod{r_j}  \end{cases}$$
Likewise, we see that
$$ F_{(\vec{s})}^{(\vec{\omega})}(x_{q+1}) = \begin{cases} c_{(\vec{s})}^{(\vec{\omega})} & \sum_{j=1}^n \frac{\ell(\vec{s})}{s_j}\omega_jd_j \equiv 0 \Mod{\ell(\vec{s})} \\ 0 & \sum_{j=1}^n \frac{\ell(\vec{s})}{s_j}\omega_jd_j \not\equiv 0 \Mod{\ell(\vec{s})}. \end{cases}$$

Thus the number of points lying over $x_{q+1}$ is
$$\sum_{\vec{s}\in\s} \sum_{\vec{\omega}\in\Omega_{\vec{s}}} \chi_{\ell(\vec{s})}\left(F_{(\vec{s})}^{(\vec{\omega})}(x_{q+1}) \right)$$
and we get the following lemma.

\begin{lem}\label{abelprojnumpts}

Let $C\in\Hh^{\vec{d}(\vec{\alpha})}$ that corresponds to $(\vec{c},(f_{\vec{\alpha}}))\in\hat{\F}_{[\vec{d}(\vec{\alpha})]}$. Then the number of projective points on the curve is
$$\#C(\Pp^1(\Ff_q)) = \sum_{x\in\Pp^1(\Ff_q)} \sum_{\vec{s}\in\s} \sum_{\vec{\omega}\in\Omega_{\vec{s}}} \chi_{\ell(\vec{s})}\left( F_{(\vec{s})}^{(\vec{\omega})}(x) \right).$$

\end{lem}

\begin{rem}

As we stated above, if $\vec{s}=(1,\dots,1)$, then $\Omega_{\vec{s}}=\{1,\dots,1)\}$ and $F_{(1,\dots,1)}^{(1,\dots,1)}(X)=1$. Hence
$$\sum_{x\in\Pp^1(\Ff_q)} \sum_{\vec{\omega}\in\Omega_{(1,\dots,1)}} \chi_{\ell(1,\dots,1)}\left( F_{(1,\dots,1)}^{(\vec{\omega})}(x) \right) = \sum_{x\in\Pp^1(\Ff_q)}1 = q+1$$
Thus,
$$\#C(\Pp^1(\Ff_q)) = q+1 + \sum_{x\in\Pp^1(\Ff_q)} \sum_{\substack{\vec{s}\in\s \\ \vec{s}\not=(1,\dots,1)}} \sum_{\vec{\omega}\in\Omega_{\vec{s}}} \chi_{\ell(\vec{s})}\left( F_{(\vec{s})}^{(\vec{\omega})}(x) \right)$$
and we get that
$$\Tr(\Frob_q) = - \sum_{x\in\Pp^1(\Ff_q)} \sum_{\substack{\vec{s}\in\s \\ \vec{s}\not=(1,\dots,1)}} \sum_{\vec{\omega}\in\Omega_{\vec{s}}} \chi_{\ell(\vec{s})}\left( F_{(\vec{s})}^{(\vec{\omega})}(x) \right)$$

\end{rem}


\section{Admissibility}\label{admis}

From now on, we fix an ordering of the elements of $\Ff_q = \{x_1,\dots,x_q\}$ and let $x_{q+1}$ denote the point at infinity of $\Pp^1(\Ff_q)$, then we have reduced the problem down to determine the size of the set
\begin{align}\label{set}
\{(\vec{c},(f_{\vec{\alpha}}))\in \hat{\F}_{[\vec{d}(\vec{\alpha})]} :  \chi_{\ell(\vec{s})}\left( F_{(\vec{s})}^{(\vec{\omega})}(x_i) \right) = \epsilon_{\vec{s},\vec{\omega},i}, \vec{s}\in\s, \vec{\omega}\in \Omega_{\vec{s}}, i=1,\dots,\ell\}
\end{align}
for some choices of $\epsilon_{\vec{s},\vec{\omega},i}\in \mu_{\ell(\vec{s})} \cup \{0\}$ and $\ell= q+1$. In fact, we will need to determine this for $\ell=0$ as well as $\ell=q+1$ in order to determine the probability. However, we will determine it for an arbitrary $\ell$.

Clearly, not all choices give a non-empty set as the polynomials $F_{(\vec{s})}^{(\vec{\omega})}$ are highly dependent on each other. This section will be devoted to determining properties of the choices of $\epsilon_{\vec{s},\vec{\omega}}$ that give a non-empty set.

\begin{defn}\label{abeladmisdef}

A set
$$\{\epsilon_{\vec{s},\vec{\omega}} \in \mu_{\ell(\vec{s})} \cup \{0\}, \vec{s}\in\s, \vec{\omega}\in \Omega_{\vec{s}}\}$$
is called \textbf{admissible} if there exists $(\vec{c},(f_{\vec{\alpha}})) \in \hat{\F}_{[\vec{d}(\vec{\alpha})]}$ and an $x\in\Pp^1(\Ff_q)$ such that
$$ \epsilon_{\vec{s},\vec{\omega}} = \chi_{\ell(\vec{s})} (F_{(\vec{s})}^{(\vec{\omega})}(x))$$
for all $\vec{s}\in\s$, $\vec{\omega}\in\Omega_{\vec{s}}$. (Note that $\epsilon_{(1,\dots,1),(1,\dots,1)}=1$.)

\end{defn}

Clearly, therefore, \eqref{set} will be non-empty if and only if the set
$$\{\epsilon_{\vec{s},\vec{\omega},i}\in\mu_{\ell(\vec{s})}\cup\{0\}, \vec{s}\in\s, \vec{\omega}\in\Omega_{\vec{s}}\}$$
is admissible for all $i$ and distinct $x_i$.

\begin{lem}\label{admis1}

For all $\vec{s}\in \s$, $\vec{\omega}\in\Omega_{\vec{s}}$ and $p|r_n$, prime, define
$$\vec{s}_p = (p^{v_p(s_1)},\dots,p^{v_p(s_n)})$$
$$\vec{\omega}_p = (\omega_1 \Mod{p^{v_p(s_1)}},\dots,\omega_n \Mod{p^{v_p(s_n)}}) \in\Omega_{\vec{s}_p}.$$
Let $m_p$ be the smallest, non-negative integer such that $m_p \equiv \ell(\vec{s}_p)^{-1} \Mod{\frac{\ell(\vec{s})}{\ell(\vec{s}_p)}}$. If $\{\epsilon_{\vec{s},\vec{\omega}}: \vec{s}\in\s, \vec{\omega}\in\Omega_{\vec{s}}\}$ is admissible then
$$\epsilon_{\vec{s},\vec{\omega}} = \prod_{p|r_n} \epsilon_{\vec{s}_p,\vec{\omega}_p}^{m_p}$$
\end{lem}

\begin{proof}

Let $\vec{s}\in \s$. If there exists a $p|r_n$, prime such that $s_j=p^{v_j}$ for $j=1,\dots,n$, then $s_{p'}=(1,\dots,1)$, $\omega_{p'}=(1,\dots,1)$ and $m_{p'}=1$ for all $p'\not= p$, making the statement trivial. Therefore, suppose there exists $\vec{s'},\vec{s''}\in \s$ such that $\vec{s'},\vec{s''}\not=(1,\dots,1)$, $s_j = s'_js''_j$ and $\gcd(\ell(\vec{s'}),\ell(\vec{s''}))=1$. (This is an analogue of writing $\vec{s}$ as a product of coprime factors).

Define $m' \equiv \ell(\vec{s'})^{-1}\Mod{\ell(\vec{s''})}$ and $m'' \equiv \ell(\vec{s''})^{-1}\Mod{\ell(\vec{s'})}$. Moreover, let
$$\vec{\omega'} = (\omega'_1,\dots,\omega'_n) = (\omega_1 \Mod{s'_1},\dots,\omega_n\Mod{s'_n}) \in \Omega_{\vec{s'}}$$
$$\vec{\omega''} = (\omega''_1,\dots,\omega''_n) = (\omega_1 \Mod{s''_1},\dots,\omega_n\Mod{s''_n}) \in \Omega_{\vec{s''}}$$

Then there exists some polynomial $H$ such that
$$(F_{(\vec{s'})}^{(\vec{\omega'})}(X))^{m''\ell(\vec{s''})}(F_{(\vec{s''})}^{(\vec{\omega''})}(X))^{m'\ell(\vec{s'})} = F_{(\vec{s})}^{(\vec{\omega})}(X) \left(H(X)\right)^{\ell(\vec{s})}$$
Moreover, all the factors that appear in $F_{(\vec{s})}^{(\vec{\omega})}(X)$ appear in either $F_{(\vec{s'})}^{(\vec{\omega'})}(X)$ or $F_{(\vec{s''})}^{(\vec{\omega''})}(X)$. That is to say, the former is zero at $x$ if and only if one of the latter are zero at $x$. Therefore,
$$\chi_{\ell(\vec{s})}\left(F_{(\vec{s})}^{(\vec{\omega})}(x)\right) = \chi^{m''}_{\ell(\vec{s'})}\left(F_{(\vec{s'})}^{(\vec{\omega'})}(x)\right)\chi^{m'}_{\ell(\vec{s''})} \left(F_{(\vec{s''})}^{(\vec{\omega''})}(x)\right)$$

Iterating this process then we get the result with the Chinese Remainder Theorem.

\end{proof}

\begin{cor}\label{adcor}

$\epsilon_{\vec{s},\vec{\omega}}$ uniquely determines and is uniquely determined by $\epsilon_{\vec{s}_p,\vec{\omega}_p}$ for all $p|r_n$.

\end{cor}

\begin{proof}

Straight forward from Lemma \ref{admis1}.

\end{proof}

\begin{lem}\label{admis2}

For any $\vec{s} = (s_1,\dots,s_n) \in\s$, define $\vec{\sigma}_j$ to be the vector in $\s$ that has $s_j$ in the $j^{th}$ coordinate and $1$ everywhere else. Let $\vec{1} = (1,\dots,1)\in\Omega_{\vec{\sigma}_j}\subset\Omega_{\vec{s}}$. If $\{\epsilon_{\vec{s},\vec{\omega}}: \vec{s}\in\s, \vec{\omega}\in\Omega_{\vec{s}}\}$ is admissible and $\epsilon_{\vec{\sigma}_j,\vec{1}} \not=0$ for all $j$ then
$$\epsilon_{\vec{s},\vec{\omega}} = \prod_{j=1}^n\epsilon^{\omega_j}_{\vec{\sigma}_j,\vec{1}}$$

\end{lem}

\begin{proof}

Recall that $F_j(X) = \prod_{\vec{\alpha}\in\R}f_{\vec{\alpha}}^{\alpha_j}(X)$. For all $s_j|r_j$ define
$$F_{j,s_j}(X) := \prod_{\vec{\alpha}\in\R}f_{\vec{\alpha}}(X)^{\alpha_j \Mod{s_j}} = F_{(\vec{\sigma}_j)}^{(\vec{1})}(X).$$

Therefore, there exists an $H$ such that
$$\prod_{j=1}^n F_{j,s_j}(X)^{\frac{\ell(\vec{s})}{s_j}\omega_j} = F_{(\vec{s})}^{(\vec{\omega})}(X) H(X)^{\ell(\vec{s})}.$$
Hence, if $F_{j,s_j}(x)\not=0$ for all $j$, then $H(x)\not=0$ and
$$\epsilon_{\vec{s},\vec{\omega}} = \chi_{\ell(\vec{s})}(F_{(\vec{s})}^{(\vec{\omega})}(x)) = \prod_{j=1}^n \chi^{\omega_j}_{s_j} (F_{j,s_j}(x)) = \prod_{j=1}^n \epsilon^{\omega_j}_{\vec{\sigma}_j,\vec{1}}.$$

\end{proof}

As in the cyclic case in \cite{M1}, it will be important to keep track of when and how an admissible set can have zero values. Fix a $\vec{\beta}$ such that $f_{\vec{\beta}}(x)=0$. Then $F_{(\vec{s})}^{(\vec{\omega})}(x)=0$ if and only if $f_{\vec{\beta}}(X)|F_{(\vec{s})}^{(\vec{\omega})}(X)$ if and only if
$$\sum_{j=1}^n \frac{\ell(\vec{s})}{s_j} \omega_j\beta_j \not\equiv 0 \Mod{\ell(\vec{s})}.$$

Define the set
$$A_{\vec{\beta}} := \{ (\vec{s},\vec{\omega}) : \vec{s}\in\s, \vec{\omega}\in\Omega_{\vec{s}}, \sum_{j=1}^n \frac{\ell(\vec{s})}{s_j} \omega_j\beta_j \equiv 0 \Mod{\ell(\vec{s})}  \}.$$
Then $F_{(\vec{s})}^{(\vec{\omega})}(x)\not=0$ if and only if $(\vec{s},\vec{\omega})\in A_{\vec{\beta}}$.

There is a natural bijective correspondence from $A_{\vec{\beta}}$ to
$$ \{\vec{\omega}\in\R^\dag : \sum_{j=1}^{n} \frac{r_n}{r_j} \omega_j\beta_j \equiv 0 \Mod{r_n} \}$$
which sends $(\vec{s},\vec{\omega}) \to (\frac{r_1}{s_1}\omega_1, \dots, \frac{r_n}{s_n}\omega_n) $ where
$$\R^\dag = [1,\dots,r_1]\times\dots\times[1,\dots,r_n].$$

We will equate the definition of $A_{\vec{\beta}}$ with this set and either talk about $(\vec{s},\vec{\omega})\in A_{\vec{\beta}}$ using the first definition or just $\vec{\omega}\in A_{\vec{\beta}}$ using the second definition depending on whichever is the most convenient.

Let $\R' = \R\cup\{(0,\dots,0)\}$ and define an equivalence relationship of $\R'$ by $\vec{\beta}\sim\vec{\beta'}$ if and only if $A_{\vec{\beta}} = A_{\vec{\beta'}}$. Let $\tilde{\R} = \R'/\sim$ and write $[\vec{\beta}]\in\tilde{\R}$ as the equivalence class of $\vec{\beta}$ in $\tilde{\R}$.

\begin{defn}\label{betadmisdef}

An admissible set
$$\{\epsilon_{\vec{s},\vec{\omega}} \in \mu_{\ell(\vec{s})} \cup \{0\}, \vec{s}\in\s, \vec{\omega}\in \Omega_{\vec{s}}\}$$
is called \textbf{$[\vec{\beta}]$-admissible} if $\epsilon_{\vec{s},\vec{\omega}}=0$ if and only if $(\vec{s},\vec{\omega})\not\in A_{\vec{\beta}}$.
\end{defn}

\begin{rem}

If $\{\epsilon_{\vec{s},\vec{\omega}} : \vec{s}\in\s,\vec{\omega}\in\Omega_{\vec{s}}\}$ is $[\vec{0}]$-admissible then  $\epsilon_{\vec{s},\vec{\omega}}\not=0$ for all $\vec{s}\in\s,\vec{\omega}\in\Omega_{\vec{s}}$.

\end{rem}

It will be useful later to classify the equivalence classes of $\tilde{\R}$. Towards this, for all $p|r_n$, define
$$\s_p = \{\vec{s}=(s_1,\dots,s_n) : s_j = p^{v_j}, 0\leq v_j \leq v_p(r_j)\} \subset \s$$
$$A_{\vec{\beta},p} := \{(\vec{s},\vec{\omega}) : \vec{s}\in\s_p, \vec{\omega}\in\Omega_{\vec{s}}, \sum_{j=1}^n \frac{\ell(\vec{s})}{s_j} \omega_j\beta_j \equiv 0 \Mod{\ell(\vec{s})} \}$$
$$ = \{\vec{\omega}\in\R^\dag_p : \sum_{j=1}^n p^{v_p(r_n)-v_p(r_j)} \omega_j\beta_j \equiv 0 \Mod{p^{v_p(r_n)}}\}$$
where we identify the two sets under the map $(\vec{s},\vec{\omega}) \to (\frac{p^{v_p(r_1)}}{s_1}\omega_1, \dots, \frac{p^{v_p(r_n)}}{s_n}\omega_n)$ and $\R^\dag_p = [1,\dots,p^{v_p(r_1)}]\times \dots \times [1,\dots,p^{v_p(r_n)}]$.

Then say $\vec{\beta}\sim_p \vec{\beta'}$ if $A_{\vec{\beta},p} = A_{\vec{\beta'},p}$. Clearly, $\vec{\beta}\sim\vec{\beta'}$ if and only if $\vec{\beta}\sim_p \vec{\beta'}$ for all $p|r_n$.

\begin{lem}\label{psimlem}

If $\vec{\beta}\sim_p\vec{\beta'}$ then $v_p((\beta_j,r_j)) = v_p((\beta'_j,r_j))$ for $j=1,\dots,n$.

\end{lem}

\begin{proof}

Let $\vec{s} = (1,\dots,1,p^{v_p((\beta_j,r_j))},1,\dots,1)$, where the $p^{v_p((\beta_j,r_j))}$ is in the $j^{th}$ coordinate. Then $(\vec{s},(1,\dots,1))\in A_{\vec{\beta},p} = A_{\vec{\beta'},p}$. This implies that
$$ \beta'_j \equiv 0 \Mod{p^{v_p((\beta_j,r_j))}}$$
And so $v_p(\beta'_j) \geq v_p((\beta_j,r_j))$. If $v_p(\beta_j) \geq v_p(r_j)$ then $v_p((\beta,r_j)) = v_p(r_j)$. Hence $v_p((\beta'_j,r_j))=r_j = v_p((\beta_j,r_j))$. If $v_p(\beta_j) < v_p(r_j)$ then $v_p(\beta'_j) \geq v_p(\beta_j)$. Similarly, we can show that $v_p(\beta_j) \geq v_p((\beta'_j,r_j))$. Thus $v_p((\beta'_j,r_j)) < v_p(r_j)$. Therefore, $v_p((\beta'_j,r_j)) = v_p(\beta'_j)$ and we get out result.

\end{proof}

\begin{lem}\label{simlem}

$\vec{\beta}\sim_p\vec{\beta'}$ if and only if there exists an $1\leq m \leq p^{\max(0, \underset{j}{\max}(v_p(\frac{r_j}{\beta_j})))}$, $(m,p)=1$ such that $\beta'_j \equiv m\beta_j \Mod{p^{v_p(r_j)}}$ for all $j$.

\end{lem}

\begin{proof}

Suppose $\vec{\beta}\sim_p\vec{\beta'}$. Then since $v_p(\beta_j,r_j)=v_p(\beta'_j,r_j)$, we can find an $m_j$ such that $1\leq m_j \leq p^{\max(0,v_p(\frac{r_j}{\beta_j}))}$, $(m_j,p)=1$ and
$$\beta'_j \equiv m_j \beta_j \Mod{p^{v_p(r_j)}}.$$
Moreover, for all $j$, define $\gamma_j$ to be such that
$$\beta_j = p^{v_p(\beta_j)}\gamma_j.$$
Let $k$ be such that $\min(v_p(\frac{r_n}{r_j}\beta_j)) = v_p(\frac{r_n}{r_k}\beta_k)$. Fix a $j$ and let $1\leq \omega_k \leq p^{v_p(r_k)}$ be smallest such that
$$\omega_k \equiv -p^{v_p(\frac{r_k\beta_j}{r_j\beta_k})}\gamma_j\gamma_k^{-1} \Mod{p^{\max(0,v_p(\frac{r_k}{\beta_k}))}}.$$
Define $\vec{\omega}\in\R^\dag_p$ such that $\omega_j=1$, $\omega_k$ is as above and $\omega_\ell=p^{v_p(r_\ell)}$ otherwise. Then $\vec{\omega}\in A_{\vec{\beta},p}=A_{\vec{\beta'},p}$. Hence,
$$ 0 \equiv p^{v_p(\frac{r_n}{r_k})}\beta'_k \omega_k + p^{v_p(\frac{r_n}{r_j})}\beta'_j \equiv p^{v_p(\frac{r_n}{r_k})}\beta_km_k \omega_k + p^{v_p(\frac{r_n}{r_j})}\beta_jm_j$$
$$ \equiv -p^{v_p(\frac{r_n}{r_k}\beta_k)}\gamma_km_kp^{v_p(\frac{r_k\beta_j}{r_j\beta_k})} \gamma_j\gamma_k^{-1} + p^{v_p(\frac{r_n}{r_j}\beta_j)}\gamma_jm_j $$
$$ \equiv p^{v_p(\frac{r_n}{r_j}\beta_j)}\gamma_j(m_j-m_k) \Mod{p^{v_p(r_n)}}$$
Therefore,
$$m_j\equiv m_k \Mod{p^{\max(0,v_p(\frac{r_j}{\beta_j}))}}.$$
Hence,
$$\beta'_j \equiv m_j \beta_j \equiv m_k \beta_j \Mod{p^{v_p(r_j)}}.$$
So, setting $m=m_k$ gives our desired result.

Conversely, suppose there exists an $1\leq m \leq p^{\max(0, \underset{j}{\max}(v_p(\frac{r_j}{\beta_j})))}$, $(m,p)=1$ such that $\beta'_j \equiv m\beta_j \Mod{p^{v_p(r_j)}}$ for all $j$. Let $\vec{\omega}\in A_{\vec{\beta},p}$. Then
$$ \sum_{j=1}^n p^{v_p(r_n)-v_p(r_j)} \omega_j\beta'_j  \equiv \sum_{j=1}^n p^{v_p(r_n)-v_p(r_j)} \omega_jm\beta_j \equiv m\sum_{j=1}^n p^{v_p(r_n)-v_p(r_j)} \omega_j\beta_j \equiv 0 \Mod{p^{v_p(r_n)}}.$$
Therefore, $\vec{\omega}\in A_{\vec{\beta'},p}$. So $A_{\vec{\beta},p}\subset A_{\vec{\beta'},p}$. However, since $(m,p)=1$, we can find an $m'$ such that $\beta_j \equiv m'\beta'_j$. From which we get $A_{\vec{\beta'},p}\subset A_{\vec{\beta},p}$ and therefore $A_{\vec{\beta},p} = A_{\vec{\beta'},p}$ and $\vec{\beta}\sim_p\vec{\beta'}$.

\end{proof}

Note that
$$\prod_{p|r_n} p^{\max(0, \underset{j}{\max}(v_p(\frac{r_j}{\beta_j})))} = \lcm\left( \frac{r_j}{(r_j,\beta_j)} \right) = e(\vec{\beta}).$$

For any natural number $m$ and $\vec{\beta}\in\R'$, define $m\vec{\beta} = (m\beta_1 \Mod{r_1}, \dots, m\beta_n \Mod{r_n})$.

\begin{cor}\label{simcor}

$\vec{\beta}\sim\vec{\beta'}$ if and only if there exists an $1\leq m \leq e(\vec{\beta})$, $(m,e(\vec{\beta}))=1$ such that $\vec{\beta'} = m\vec{\beta}$.

\end{cor}

\begin{proof}

Suppose $\vec{\beta}\sim\vec{\beta'}$. Then $\vec{\beta}\sim_p\vec{\beta'}$ for all $p|r_n$ and we can find an $1\leq m_p \leq p^{\max(0, \underset{j}{\min}(v_p(\frac{r_n}{r_j}\beta_j)))}$, $(m_p,p)=1$ such that $\beta'_j \equiv m_p\beta_j \Mod{p^{v_p(r_j)}}$. Let $1\leq m \leq \prod_{p|r_n} p^{\max(0, \underset{j}{\min}(v_p(\frac{r_n}{r_j}\beta_j)))}$, $(m,r_n)=1$ such that $m\equiv m_p \Mod{p^{\max(0, \underset{j}{\min}(v_p(\frac{r_n}{r_j}\beta_j)))}}$ for all $p|r_n$. Then $\beta'_j \equiv m\beta_j \Mod{r_j}$ and $\vec{\beta'}=m\vec{\beta}$

Conversely, suppose such an $m$ exists. Then let $m_p \equiv m \Mod{p^{\max(0, \underset{j}{\min}(v_p(\frac{r_n}{r_j} \beta_j)))}}$. Then $\beta_j \equiv m_p \beta'_j \Mod{p^{v_p(r_n)}}$. Thus $\vec{\beta}\sim_p\vec{\beta'}$ for all $p$ and therefore $\vec{\beta}\sim\vec{\beta'}$.

\end{proof}

\begin{cor}\label{simcor2}

There are $\phi(e(\vec{\beta}))$ different $\vec{\beta'}$ such that $\vec{\beta'}\sim\vec{\beta}$.

\end{cor}

\begin{proof}

It is easy to see that, by construction, all the $m\vec{\beta}$ are distinct for $1\leq m \leq e(\vec{\beta})$, $(m,e(\vec{\beta}))=1$.

\end{proof}

\begin{lem}\label{sizelem}

$|A_{\vec{\beta},p}| = p^{v_p(|G|)-v_p(e(\vec{\beta}))}$

\end{lem}

\begin{proof}

Consider the map
\begin{center}
\begin{tabular}{ c c c c }
$\phi_{\vec{\beta}}$ : & $\Z/p^{v_p(r_1)}\Z\times\dots\times\Z/p^{v_p(r_n)}\Z$ & $\to$ & $\Z/p^{v_p(r_n)}\Z$\\
& $(x_1,\dots,x_n)$ & $\to$ & $\sum_{j=1}^n p^{v_p(r_n)-v_p(r_j)}\beta_jx_j$
\end{tabular}
\end{center}

Then $A_{\vec{\beta},p} = \ker(\phi_{\vec{\beta}})$. Let $1\leq k\leq n$ such that $v_p(\frac{r_n}{r_k}\beta_k) = \min(v_p(\frac{r_n}{r_j}\beta_j)) $. Then Im$(\phi_{\vec{\beta}}) \subset \Z/p^{\max(0,v_p(\frac{r_k}{\beta_k}))}\Z$. Moreover
$$\phi_{\vec{\beta}}(0,\dots,0,x_k,0,\dots,0) = p^{v_p(\frac{r_n}{r_k}\beta_k)}\gamma_kx_k$$
where $\beta_k = p^{v_p(\beta_k)}\gamma_k$. Therefore, since $(\gamma_k,p)=1$, we get that Im$(\phi_{\vec{\beta}}) = \Z/p^{\max(0,v_p(\frac{r_k}{\beta_k}))}\Z$. Hence
$$|A_{\vec{\beta},p}| = |\ker(\phi_{\vec{\beta}})| = \frac{p^{v_p(|G|)}}{|\mbox{Im}(\phi_{\vec{\beta}})|} = p^{v_p(|G|)-\max(0,v_p(\frac{r_k}{\beta_k}))} = p^{v_p(|G|)-v_p(e(\vec{\beta}))}.$$

\end{proof}

\begin{cor}\label{sizelemcor}

$|A_{\vec{\beta}}| = \frac{|G|}{e(\vec{\beta})}$

\end{cor}

\begin{proof}

$$|A_{\vec{\beta}}| = \prod_{p|r_n}|A_{\vec{\beta},p}| = \prod_{p|r_n}p^{v_p(|G|)-v_p(e(\vec{\beta}))} =\frac{|G|}{e(\vec{\beta})}.$$

\end{proof}


\section{Value Taking}\label{abelvaltak}

In this section we will determine for $0\leq \ell \leq q$, the size of the set
\begin{align}\label{Fsetabel1}
\{(f_{\vec{\alpha}})\in\F_{\vec{d}(\vec{\alpha})} : \chi_{\ell(\vec{s})}(F_{(\vec{s})}^{(\vec{\omega})}(x_i)) = \epsilon_{\vec{s},\vec{\omega},i}, \vec{s}\in\s, \vec{\omega}\in\Omega_{\vec{s}}, i=1,\dots,\ell\}
\end{align}
where for $i=1,\dots,\ell$, the set
\begin{align}\label{admisset}
\{\epsilon_{\vec{s},\vec{\omega},i}\in\mu_{\ell(\vec{s})}\cup\{0\}: \vec{s}\in\s,\vec{\omega}\in\Omega_{\vec{s}}\}
\end{align}
is admissible.

\begin{rem}

Since we are assuming for now that $\ell \leq q$, we are only dealing with the affine points, hence we need only look at the set $\F_{\vec{d}(\vec{\alpha})}$. If we want to incorporate the point at infinity by setting $\ell=q+1$, we need to consider the full set $\hat{\F}_{[\vec{d}(\vec{\alpha})]}$ as will be done in Proposition \ref{valtakpropinf}.

\end{rem}

Define $\vec{\rho}_j = (1,\dots,r_j,\dots,1)\in\s$ where the $r_j$ is in the $j^{th}$ coordinate. Denote $\vec{1}=(1,\dots,1)$. Then
$$F_j(X) := \prod_{\vec{\alpha}\in\R}f_{\vec{\alpha}}^{\alpha_j} = F_{(\vec{\rho}_j)}^{(\vec{1})}(X).$$

By Lemmas \ref{admis1} and \ref{admis2}, we get that if $\epsilon_{\vec{s},\vec{\omega},i}\not=0$ for all $\vec{s}\in\s$, $\vec{\omega}\in\Omega_{\vec{s}}$ and $i=1,\dots,\ell$, then the values of $\epsilon_{\vec{s},\vec{\omega},i}$ will be uniquely determined by the values of $\epsilon_{\vec{\rho}_j,\vec{1},i}$ for $j=1,\dots,n$, $i=1,\dots,\ell$. Moreover, \eqref{admisset} will be admissible for any choices of $\epsilon_{\vec{\rho_j}, \vec{1},i}$. Therefore,
$$|\{(f_{\vec{\alpha}})\in\F_{\vec{d}(\vec{\alpha})} : \chi_{\ell(\vec{s})}(F_{(\vec{s})}^{(\vec{\omega})}(x_i)) = \epsilon_{\vec{s},\vec{\omega},i}, \vec{s}\in\s, \vec{\omega}\in\Omega_{\vec{s}}, i=1,\dots,\ell\}|$$
$$ = |\{(f_{\vec{\alpha}})\in\F_{\vec{d}(\vec{\alpha})} : \chi_{r_j}(F_j(x_i)) = \epsilon_{\vec{\rho}_j,\vec{1},i}, i=1\dots,\ell, j=1,\dots,n \}|$$

The size of this set is easily deduced from Proposition 3.4 of \cite{M1}.

\begin{prop}\label{keyprop}

Let $\vec{d}(\vec{\alpha})$ be as above. For $0\leq\ell\leq q$ let $\epsilon_{\vec{\rho_j},\vec{1},i}\in\mu_{r_j}$ for $i=1,\dots,\ell$. Then the size of
\begin{align*}
\{(f_{\vec{\alpha}}) \in \F_{\vec{d}(\vec{\alpha})} : \chi_{r_j}\left(F_{(\vec{\rho_j})}^{(\vec{1})}(x_i)\right) = \epsilon_{\vec{\rho_j},\vec{1},i}, 1\leq i \leq \ell, 1 \leq j \leq n\}
\end{align*}
is
\begin{align*}
S_n(\ell) := \frac{L_{|G|-2} q^{\sum d({\vec{\alpha}})}}{\zeta_q(2)^{|G|-1}} \left(\frac{q}{|G|(q+|G|-1)}\right)^\ell\left(1 + O\left(q^{-\frac{\min( d({\vec{\alpha}}))}{2}}\right)\right)
\end{align*}
where $\zeta_q(s)$ is the zeta function for $K$ and
$$L_n = \prod_{j=1}^n \prod_{P}\left(1 - \frac{j}{(|P|-1)(|P|+j)}\right)$$
where the product is over all monic, irreducible polynomials of $K$ and $|P|=q^{\deg(P)}$.

\end{prop}

\begin{rem}

First, note that $|G|=r_1\cdots r_n$. Moreover, Proposition 3.4 of \cite{M1} does not rely on the fact that the $r_j|r_{j+1}$ and hence define a group. Secondly, observe that the size of the set is independent of the choices of $\epsilon_{\vec{\rho}_j,\vec{1},i}$ as long as they are non-zero.

\end{rem}

\begin{rem}

The error term is written in terms of $\min(d(\vec{\alpha}))$ and so is only smaller than the main term if $\min(d(\vec{\alpha}))$ tends to infinity. This is equivalent to saying that all the $d(\vec{\alpha})$ tend to infinity. This calculation is why we need that assumption in the Theorem \ref{mainthm2} and why we can not easily extend this result to the whole space $\Hh_{G,g}$. Therefore, improving this error term is one way in which we could extend the result however, this seems unlikely. A different method for doing this is the topic of a forthcoming paper by the author. 

\end{rem}

\begin{cor}\label{keypropcor}

$$|\hat{\F}_{[\vec{d}(\vec{\alpha})]}| = \frac{(q-1)^n(q+|G|-1)}{q}\frac{L_{|G|-2} q^{\sum d({\vec{\alpha}})}}{\zeta_q(2)^{|G|-1}} \left(1 + O\left(q^{-\frac{\min( d({\vec{\alpha}}))}{2}}\right)\right)$$

\end{cor}

\begin{proof}

This is straight forward from setting $\ell=0$ in Proposition \ref{keyprop} and summing up over the components of $\hat{\F}_{[\vec{d}(\vec{\alpha})]}$ and choices of $\vec{c}\in(\Ff_q^*)^n$.

\end{proof}

Let us now determine the size of the set if some of the $\epsilon_{\vec{s},\vec{\omega},i}$ can be zero. With the notation of Section \ref{admis}, we would need the set in \ref{admisset} to be $[\vec{\beta}]$-admissible for some $[\vec{\beta}]\in\tilde{\R}$.

\begin{prop}\label{valtakprop}

Let $\{\epsilon_{\vec{s},\vec{\omega},i} : \vec{s}\in\s,\vec{\omega}\in\Omega_{\vec{s}}\}$ be an admissible set for $1\leq i \leq \ell$ such that
$$m_{[\vec{\beta}]} := |\{1\leq i \leq \ell :  \{\epsilon_{\vec{s},\vec{\omega},i} : \vec{s}\in\s,\vec{\omega}\in\Omega_{\vec{s}}\} \mbox{ is } [\vec{\beta}]-\mbox{admissible} \}|$$
then
$$|\{(f_{\vec{\alpha}})\in\F_{\vec{d}(\vec{\alpha})} : \chi_{\ell(\vec{s})}(F_{(\vec{s})}^{(\vec{\omega})}(x_i)) = \epsilon_{\vec{s},\vec{\omega},i}, \vec{s}\in\s, \vec{\omega}\in\Omega_{\vec{s}}, i=1,\dots,\ell\}|$$
$$ = \frac{L_{|G|-2}q^{\sum d(\vec{\alpha})}}{\zeta_q(2)^{|G|-1}} \prod_{\substack{[\vec{\beta}]\in\tilde{\R} \\ [\vec{\beta}]\not=[\vec{0}]}} \left( \frac{\phi(e(\vec{\beta})^2)} {|G|(q+|G| -1)} \right)^{m_{[\vec{\beta}]}}\left( \frac{q}{|G|(q+|G|-1)} \right)^{m_{[\vec{0}]}}\left(1 + O\left(q^{-\frac{\min(d(\vec{\alpha}))}{2}}\right)\right).$$

\end{prop}

\begin{rem}

If all the $\epsilon_{\vec{s},\vec{\omega},i}\not=0$, then this implies $m_{[0]}=\ell$ and $m_{[\vec{\beta}]}=0$ and we get back the result of Proposition \ref{keyprop}.

\end{rem}

\begin{proof}

For every $[\vec{\beta}]\in\tilde{\R}$, define
$$M_{[\vec{\beta}]} = \{1\leq i \leq \ell : \{\epsilon_{\vec{d},\vec{i},i} :  \vec{d}\in D, \vec{i} \in I_{\vec{d}}\} \mbox{ is $[\vec{\beta}]$-admissible}\}$$
Then $m_{[\vec{\beta}]} = |M_{[\vec{\beta}]}|$ and
$$\sum_{[\vec{\beta}]\in\tilde{\R}}m_{[\vec{\beta}]} = \ell.$$
Moreover, if $i\in M_{[\vec{\beta}]}$ for $\vec{\beta}\not=\vec{0}$ then $f_{\vec{\beta'}}(x_i)=0$ for some $\vec{\beta'} \sim \vec{\beta}$.

For all $\vec{\beta}\not=\vec{0}$, fix a partition of $M_{[\vec{\beta}]}$ as
$$M_{[\vec{\beta}]} = \bigcup_{\vec{\beta'}\sim\vec{\beta}} M_{\vec{\beta'}} =  \bigcup_{\vec{\beta'}\sim\vec{\beta}} \{1 \leq i \leq \ell: f_{\vec{\beta'}}(x_i)=0\}$$
and let $m_{\vec{\beta}} = |M_{\vec{\beta}}|$.

For all $\vec{\beta}\in\R$ define $g_{\vec{\beta}}(X)$ as
$$f_{\vec{\beta}}(X) = g_{\vec{\beta}}(X)\prod_{i\in M_{\vec{\beta}}}(X-x_i).$$
Likewise, define $G_{(\vec{s})}^{(\vec{\omega})}(X)$ as the corresponding products of the $g_{\vec{\alpha}}(X)$. Recall, for any $\vec{s}\in\s$, we let $\vec{\sigma}_j\in\s$ be the vector that has $s_j$ in the $j^{th}$ coordinate and $1$ everywhere else. Further, for any $s_j|r_j$,
$$G_{j,s_j}(X) := \prod_{\vec{\alpha}\in\R} g_{\vec{\alpha}}(X)^{\alpha_j \Mod{s_j}} = G_{(\vec{\sigma}_j)}^{(\vec{1})}(X)$$
where we use the convention that $G_{j,1}(X)=1$.

Since $g_{\vec{\alpha}}(x_i)\not=0$ for all $1\leq i \leq \ell$ we get that $G_{(\vec{s})}^{(\vec{\omega})}(x_i)\not=0$ for all $1\leq i \leq \ell$ and hence, by Lemmas \ref{admis1} and \ref{admis2}, $\chi_{\ell(\vec{s})}\left(G_{(\vec{s})}^{(\vec{\omega})}(x_i)\right)$ will be determined by $\chi_{r_j}\left(G_{j,r_j}(x_i)\right)$, for $j=1,\dots,n$. Moreover,by Corollary \ref{adcor} these will be determined by
$$\chi_{p^{v_p(r_j)}}\left(G_{j,p^{v_p(r_j)}}(x_i)\right) \mbox{ for all $p|r_n$, $j=1,\dots,n$ }.$$

Now fix an $i\in M_{\vec{\beta}}$. If $(\vec{s},\vec{\omega})\in A_{\vec{\beta}}$, then
$$F_{(\vec{s})}^{(\vec{\omega})}(X) = G_{(\vec{s})}^{(\vec{\omega})}(X) H(X)$$
for some $H(X)$ such that $H(x_i)\not=0$. Moreover, $H(X)$ depends only on the choice of partitions of the $M_{[\vec{\beta}]}$. Therefore, for a fixed partition, we see that $\chi_{\ell(\vec{s})}(G_{(\vec{s})}^{(\vec{\omega})}(x_i))$ will be determined by $\chi_{\ell(\vec{s})}(F_{(\vec{s})}^{(\vec{\omega})}(x_i))$ for all $(\vec{s},\vec{\omega})\in A_{\vec{\beta}}$. It remains to determine how many choices there are for $\chi_{\ell(\vec{s})}(G_{(\vec{s})}^{(\vec{\omega})}(x_i))$ such that $(\vec{s},\vec{\omega})\not\in A_{\vec{\beta}}$.

Fix a $p|r_n$ and let $k$ be such that
$$\min \left(v_p\left(\frac{r_n}{r_j}\beta_j\right)\right) = v_p\left(\frac{r_n}{r_k}\beta_k\right)$$
Then I claim that if we know $\chi_{p^{v_p(r_k)}}(G_{k,p^{v_p(r_k)}}(x_i))$ then we know $\chi_{p^{v_p(r_j)}}(G_{j,p^{v_p(r_j)}}(x_i))$ for all $1\leq j \leq n$. If we write $\beta_j = p^{v_p(\beta_j)}\gamma_j$, $r_j=p^{v_p(r_j)}s_j$ where $(\gamma_j,p)=(s_j,p)=1$ and let
$$\omega'_k \equiv \gamma_k^{-1}\gamma_j p^{v_p\left(\frac{r_k\beta_j}{r_j\beta_k}\right)} \Mod{p^{\max(v_p\left(\frac{r_k}{\beta_k}\right),0)}}$$
then we see that
$$\frac{r_n}{r_k}\beta_k\omega'_ks_k + \frac{r_n}{r_j}\beta_js_j \equiv 0 \Mod{r_n}$$
Therefore, defining $\vec{\omega}\in\R^\dag$ as $\omega_j=s_j$, $\omega_h=r_h$, $h\not=j,k$ and $\omega_k=\omega'_ks_k$, then $\vec{\omega}\in A_{\vec{\beta}}$. So, defining $\vec{p} = (p^{v_p(r_1)},\dots,p^{v_p(r_n)})$ we get by Lemma \ref{admis2},
$$\chi_{p^{v_p(r_n)}}\left(G_{(\vec{p})}^{(\vec{\omega})}(x_i)\right) = \chi_{p^{v_p(r_k)}}^{\omega'_k} \left(G_{k,p^{v_p(r_k)}}(x_i)\right) \chi_{p^{v_p(r_j)}}\left(G_{j,p^{v_p(r_j)}}(x_i)\right)$$
Moreover, as stated above, $\chi_{p^{v_p(r_n)}}\left(G_{(\vec{p})}^{(\vec{\omega})}(x_i)\right)$ is fixed by $\chi_{p^{v_p(r_n)}}\left(F_{(\vec{p})}^{(\vec{\omega})}(x_i)\right)$ and our choices of $M_{\vec{\beta}}$. Hence knowing $\chi_{p^{v_p(r_k)}}(G_{k,p^{v_p(r_k)}}(x_i))$ fixes $\chi_{p^{v_p(r_j)}}(G_{j,p^{v_p(r_j)}}(x_i))$.

Therefore, to determine the number of possible values for $\chi_{p^{v_p(r_j)}}(G_{j,p^{v_p(r_j)}}(x_i))$, $j=1,\dots,n$, it is enough to determine the possible values for $\chi_{p^{v_p(r_k)}}(G_{k,p^{v_p(r_k)}}(x_i))$.

Finally, since $\chi_{p^{v_p(\beta_k)}}(G_{k,p^{v_p(\beta_k)}}(x_i))$ is determined by $\chi_{p^{v_p(\beta_k)}}(F_{k,p^{v_p(\beta_k)}}(x_i))$ and the choice of $M_{\vec{\beta}}$ there are $p^{\max(v_p\left(\frac{r_k}{\beta_k}\right),0)}$ choices for $\chi_{p^{v_p(r_k)}}(G_{k,p^{v_p(r_k)}}(x_i))$.

All together, therefore, there are
$$\prod_{p|r_n} p^{\max(0,\max_j(v_p(\frac{r_j}{\beta_j})))} = \underset{j=1,\dots,n}{\lcm}\left(\frac{r_j}{(r_j,\beta_j)}\right) = e(\vec{\beta})$$
different choices for
$$\chi_{p^{v_p(r_j)}}\left(G_{j,p^{v_p(r_j)}}(x_i)\right) \mbox{ for all $p|r_n$, $j=1,\dots,n$ }$$
and hence $e(\vec{\beta})$ different choices for
$$\chi_{\ell(\vec{s})}(G_{(\vec{s})}^{(\vec{\omega})}(x_i)), \vec{s}\in\s, \vec{\omega}\in\Omega_{\vec{s}}$$
for a fixed choice of the $M_{\vec{\beta}}$

Therefore,
$$|\{(f_{\vec{\alpha}})\in\F_{\vec{d}(\vec{\alpha})} : \chi_{\ell(\vec{s})}(F_{(\vec{s})}^{(\vec{\omega})}(x_i)) = \epsilon_{\vec{s},\vec{\omega},i}, \vec{s}\in\s, \vec{\omega}\in\Omega_{\vec{s}}, i=1,\dots,\ell\}|$$
$$ = \sum_{M_{\vec{\beta}}}\sum_{\epsilon'_{\vec{s},\vec{\omega},i}} |\{(g_{\vec{\alpha}})\in\F_{\vec{d'}(\vec{\alpha})} : \chi_{\ell(\vec{s})}(G_{(\vec{s})}^{(\vec{\omega})}(x_i)) = \epsilon'_{\vec{s},\vec{\omega},i}, \vec{s}\in\s, \vec{\omega}\in\Omega_{\vec{s}}, i=1,\dots,\ell\}|$$
where the first sum is over all the partitions $M_{[\vec{\beta}]} = \bigcup_{\vec{\beta}\sim\vec{\beta'}} M_{\vec{\beta}}$, the second sum is over all $e(\vec{\beta})$ choices of $\chi_{\ell(\vec{s})}(G_{(\vec{s})}^{(\vec{\omega})}(x_i))$ and $\vec{d'}(\vec{\alpha})$ is the vector such that $d'(\vec{\alpha}) = d(\vec{\alpha})-m_{\vec{\alpha}}$. Now since $\epsilon'_{\vec{s},\vec{\omega},i}\not=0$ for all $\vec{s}\in\s, \vec{\omega}\in\Omega_{\vec{s}}$, $i=1,\dots,\ell$, we get by Proposition \ref{keyprop}, the above line is equal to
$$ \sum_{M_{\vec{\beta}}}\sum_{\epsilon'_{\vec{s},\vec{\omega},i}} \frac{L_{|G|-2}q^{\sum d'(\vec{\alpha})}}{\zeta_q(2)^{|G|-1}} \left(\frac{q}{|G|(q+|G|-1)}\right)^{\ell}\left(1 + O\left(q^{-\frac{\min(d(\vec{\alpha}))}{2}}\right)\right) $$
$$= \sum_{M_{\vec{\beta}}} \prod_{\substack{[\vec{\beta}]\in\tilde{\R} \\ [\vec{\beta}]\not=[\vec{0}]}} e(\vec{\beta})^{m_{[\vec{\beta}]}} \frac{L_{|G|-2}q^{\sum (d(\vec{\alpha}) - m_{\vec{\alpha}})}}{\zeta_q(2)^{|G|-1}} \left(\frac{q}{|G|(q+|G|-1)}\right)^{\ell}\left(1 + O\left(q^{-\frac{\min(d(\vec{\alpha}))}{2}}\right)\right)$$
$$=  \frac{L_{|G|-2}q^{\sum d(\vec{\alpha})}}{\zeta_q(2)^{|G|-1}} \sum_{M_{\vec{\beta}}} \prod_{\substack{[\vec{\beta}]\in\tilde{\R} \\ [\vec{\beta}]\not=[\vec{0}]}} \left(\frac{e(\vec{\beta})}{|G|(q+|G|-1)}\right)^{m_{[\vec{\beta}]}} \left(\frac{q}{|G|(q+|G|-1)}\right)^{m_{[\vec{0}]}}\left(1 + O\left(q^{-\frac{\min(d(\vec{\alpha}))}{2}}\right)\right)$$
$$=  \frac{L_{|G|-2}q^{\sum d(\vec{\alpha})}}{\zeta_q(2)^{|G|-1}} \prod_{\substack{[\vec{\beta}]\in\tilde{\R} \\ [\vec{\beta}]\not=[\vec{0}]}} \left(\frac{\phi(e(\vec{\beta})^2)}{|G|(q+|G|-1)}\right)^{m_{[\vec{\beta}]}} \left(\frac{q}{|G|(q+|G|-1)}\right)^{m_{[\vec{0}]}}\left(1 + O\left(q^{-\frac{\min(d(\vec{\alpha}))}{2}}\right)\right)$$
where the last equality comes from Corollary \ref{adcor} that states that there are $\phi(e(\vec{\beta}))$ different $\vec{\beta'}$ such that $\vec{\beta'}\sim\vec{\beta}$.

\end{proof}

Recall that $x_{q+1}$ is the point at infinity and if $(\vec{c},(f_{\vec{\alpha}}))\in\hat{\F}_{[\vec{d}(\vec{\alpha})]}$, then
$$ F_{(\vec{s})}^{(\vec{\omega})}(x_{q+1}) = \begin{cases} 0 & \sum_{j=1}^n \frac{\ell(\vec{s})}{s_j}\omega_jd_j \not\equiv 0 \Mod{\ell(\vec{s})} \\ c_{(\vec{s})}^{(\vec{\omega})} & \sum_{j=1}^n \frac{\ell(\vec{s})}{s_j}\omega_jd_j \equiv 0 \Mod{\ell(\vec{s})} \end{cases}.$$

\begin{prop}\label{valtakpropinf}

Let $\{\epsilon_{\vec{s},\vec{\omega},i} : \vec{s}\in\s,\vec{\omega}\in\Omega_{\vec{s}}\}$ be an admissible set for $1\leq i \leq q+1$ such that
$$m_{[\vec{\beta}]} := |\{1\leq i \leq q+1 :  \{\epsilon_{\vec{s},\vec{\omega},i} : \vec{s}\in\s,\vec{\omega}\in\Omega_{\vec{s}}\} \mbox{ is } [\vec{\beta}]-\mbox{admissible} \}|$$
then
$$|\{(\vec{c},(f_{\vec{\alpha}}))\in\hat{\F}_{[\vec{d}(\vec{\alpha})]} : \chi_{\ell(\vec{s})}(F_{(\vec{s})}^{(\vec{\omega})}(x_i)) = \epsilon_{\vec{s},\vec{\omega},i}, \vec{s}\in\s, \vec{\omega}\in\Omega_{\vec{s}}, i=1,\dots,q+1\}|$$
$$ = \frac{(q-1)^n(q+|G|-1)}{q}\frac{L_{|G|-2}q^{\sum d(\vec{\alpha})}}{\zeta_q(2)^{|G|-1}} \prod_{\substack{[\vec{\beta}]\in\tilde{\R} \\ [\vec{\beta}]\not=[\vec{0}] }} \left( \frac{\phi(e(\vec{\beta})^2)} {|G|(q+|G| -1)} \right)^{m_{[\vec{\beta}]}}\left( \frac{q}{|G|(q+|G|-1)} \right)^{m_{[\vec{0}]}}\times$$
$$\left(1 + O\left(q^{-\frac{\min(d(\vec{\alpha}))}{2}}\right)\right).$$

\end{prop}

\begin{rem}

Notice that we are looking at $(\vec{c},(f_{\vec{\alpha}}))\in\hat{\F}_{[\vec{d}(\vec{\alpha})]}$. That is, when we add in the point at infinity, we must consider the whole irreducible coarse moduli space.

\end{rem}

\begin{proof}

\textbf{Case 1:} $\epsilon_{\vec{s},\vec{\omega},q+1}\not=0$ for all $\vec{s}\in \s, \vec{\omega}\in \Omega_{\vec{s}}$

This means that $(\vec{c},(f_{\vec{\alpha}}))\in\hat{\F}_{\vec{d}(\vec{\alpha})}$ and $\chi_{\ell(\vec{s})}(F_{(\vec{s})}^{(\vec{\omega})} (x_{q+1}))$ will be determine by $\chi_{r_j}(F_j(x_{q+1}))$, $j=1,\dots,n$. Moreover, $\chi_{r_j}(c_j) = \chi_{r_j}(F_j(x_{q+1}))$, so $c_j$ has $(q-1)/r_j$ choices for all $j$. That is
$$|\{(\vec{c},(f_{\vec{\alpha}}))\in\hat{\F}_{[\vec{d}(\vec{\alpha})]} : \chi_{\ell(\vec{s})}(F_{(\vec{s})}^{(\vec{\omega})}(x_i)) = \epsilon_{\vec{s},\vec{\omega},i}, \vec{s}\in \s, \vec{\omega}\in \Omega_{\vec{s}}, 1\leq i \leq q+1 \}|$$
$$ = \sum_{c_j} |\{(f_{\vec{\alpha}})\in\F_{\vec{d}(\vec{\alpha})} : \chi_{\ell(\vec{s})}(F_{(\vec{s})}^{(\vec{\omega})}(x_i)) = \epsilon_{\vec{s},\vec{\omega},i}, \vec{s}\in \s, \vec{\omega}\in \Omega_{\vec{s}}, 1\leq i \leq q \}|$$
$$= \sum_{c_j} \frac{L_{|G|-2} q^{\sum d(\vec{\alpha})}}{\zeta_q(2)^{|G|-1}} \prod_{\substack{[\vec{\beta}]\in\tilde{\R} \\ [\vec{\beta}]\not=[\vec{0}]}} \left( \frac{\phi(e(\vec{\beta})^2)}{|G|(q+|G| -1)} \right)^{m_{[\vec{\beta}]}} \left( \frac{q}{|G|(q+|G|-1)} \right)^{m_{[\vec{0}]}-1}\left(1 + O\left(q^{-\frac{\min(d(\vec{\alpha}))}{2}}\right)\right) $$
$$ = \frac{(q-1)^n(q+|G|-1)}{q}\frac{L_{|G|-2} q^{\sum d(\vec{\alpha})}}{\zeta_q(2)^{|G|-1}}\prod_{\substack{[\vec{\beta}]\in \tilde{\R} \\ [\vec{\beta}]\not=[\vec{0}]}} \left( \frac{\phi(e(\vec{\beta})^2)} {|G|(q+|G| -1)} \right)^{m_{[\vec{\beta}]}}\left( \frac{q}{|G|(q+|G|-1)} \right)^{m_{[\vec{0}}]}\times$$
$$\left(1 + O\left(q^{-\frac{\min(d(\vec{\alpha}))}{2}}\right)\right)$$
where the sum is over all $c_j$ such that $\chi_{r_j}(c_j)=\chi_{r_j}(F_j(x_{q+1}))$.

\textbf{Case 2:} the set $\{\epsilon_{\vec{s},\vec{\omega},q+1}: \vec{s}\in \s, \vec{\omega} \in \Omega_{\vec{s}} \}$ is $[\vec{\beta}]$-admissible for some $[\vec{\beta}]\in\tilde{\R}$, $[\vec{\beta}]\not=[\vec{0}]$.

This means that $\deg(F_j)\equiv \beta'_j \Mod{r_j}$ for some $\vec{\beta'}\sim\vec{\beta}$ and that $(\vec{c},(f_{\vec{\alpha}}))\in\hat{\F}^{\vec{\beta'}}_{\vec{d}(\vec{\alpha})}$. Fix a $p|r_n$ and let $k$ be such that $\max(v_p(\frac{r_n}{r_j}\beta'_j)) = v_p(\frac{r_n}{r_k}\beta'_k)$. Then $\chi_{p^{v_p(\beta_k)}}(F_{k,p^{v_p(\beta_k)}}(x_{q+1})) = \chi_{p^{v_p(\beta_k)}}(c_k)$. So $c_k$ has $\frac{q-1}{p^{v_p(\beta_k)}}$ choices.

Now suppose $\beta'_j = p^{b_j}\gamma_j$ and let $\omega_k$ be such that
$$ \omega_k \equiv \gamma_k^{-1}\gamma_j p^{v_p(r_k\beta_j/r_j\beta_k)} \Mod{p^{v_p(r_k/\beta_k)}}.$$
then $\chi_{p^{v_p(r_j)}}(c_k^{\omega_k}c_j)\not=0$ will be fixed. Therefore, for a choice of $c_k$ there are $\frac{q-1}{p^{v_p(r_j)}}$ choices for $c_j$ that satisfy this property.

Likewise for another $p'|r_n$, $p\not=p'$, let $k'$ be such that $\max(v_{p'}(\frac{r_n}{r_j}\beta'_j)) = v_{p'}(\frac{r_n}{r_{k'}}\beta'_{k'})$. Then the number of choices for $c_{k'}$ will be divided by $(p')^{v_{p'}(\beta'_{k'})}$ whereas the number of choice for $c_j$, $j\not=k'$ will be divided by $(p')^{v_{p'}(r_j)}$. Hence, the number of choices for the $c_j$ will be
$$\frac{(q-1)^n}{\prod_{p|r_n} \left(p^{v_p(\beta_k)}\prod_{j\not=k}p^{v_p(r_j)}\right)} = e(\vec{\beta}) \prod_{j=1}^n \frac{(q-1)}{r_j}$$

Moreover, $m_{[\vec{\beta}]}$ goes to $m_{[\vec{\beta}]}-1$. So,

$$|\{(\vec{c},(f_{\vec{\alpha}}))\in\hat{\F}_{[\vec{d}(\vec{\alpha})]} : \chi_{\ell(\vec{s})}(F_{(\vec{s})}^{(\vec{\omega})}(x_i)) = \epsilon_{\vec{s},\vec{\omega},i}, \vec{s}\in \s, \vec{\omega}\in \Omega_{\vec{s}}, 1\leq i \leq q+1 \}|$$
$$ = e(\vec{\beta})\prod_{j=1}^n\frac{(q-1)}{r_j} \sum_{\vec{\beta'}\sim\vec{\beta}} |\{(f_{\vec{\alpha}})\in\F^{\vec{\beta'}}_{\vec{d}(\vec{\alpha})} : \chi_{\ell(\vec{s})}(F_{(\vec{s})}^{(\vec{\omega})}(x_i)) = \epsilon_{\vec{s},\vec{\omega},i}, \vec{s}\in \omega, \vec{\omega}\in \Omega_{\vec{s}}, 1\leq i \leq q\}| $$
$$ = e(\vec{\beta})\prod_{j=1}^n\frac{(q-1)}{r_j} \sum_{\vec{\beta'}\sim\vec{\beta}} \frac{L_{|G|-2} q^{\sum d(\vec{\alpha})-1}}{\zeta_q(2)^{|G|-1}} \prod_{\substack{[\vec{\beta}]\in\tilde{\R} \\ [\vec{\beta}]\not=[\vec{0}]}} \left( \frac{\phi(e(\vec{\beta})^2)}{|G|(q+|G| -1)} \right)^{m_{[\vec{\beta}]}} \times$$
$$\left( \frac{\phi(e(\vec{\beta})^2)}{|G|(q+|G| -1)} \right)^{-1}\left( \frac{q}{|G|(q+|G|-1)} \right)^{m_{\vec{0}}}\left(1 + O\left(q^{-\frac{\min(d(\vec{\alpha}))}{2}}\right)\right)$$
$$ = \frac{(q-1)^n(q+|G|-1)}{q}\frac{L_{|G|-2} q^{\sum d(\vec{\alpha})}}{\zeta_q(2)^{|G|-1}}\prod_{\substack{[\vec{\beta}] \in\tilde{\R} \\ [\vec{\beta}]\not=[\vec{0}]}} \left( \frac{\phi(e(\vec{\beta})^2)} {|G|(q+|G| -1)} \right)^{m_{[\vec{\beta}]}}\left( \frac{q}{|G|(q+|G|-1)} \right)^{m_{\vec{0}}}\times$$
$$\left(1 + O\left(q^{-\frac{\min(d(\vec{\alpha}))}{2}}\right)\right).$$

Therefore, regardless of what happens at $x_{q+1}$, we get the same result.

\end{proof}

\begin{cor}\label{valtakcor}

Let $\{\epsilon_{\vec{s},\vec{\omega},i} : \vec{s}\in\s,\vec{\omega}\in\Omega_{\vec{s}}\}$ be an admissible set for $1\leq i \leq q+1$ such that
$$m_{[\vec{\beta}]} := |\{1\leq i \leq q+1 :  \{\epsilon_{\vec{s},\vec{\omega},i} : \vec{s}\in\s,\vec{\omega}\in\Omega_{\vec{s}}\} \mbox{ is } [\vec{\beta}]-\mbox{admissible} \}|$$
then
$$\frac{|\{(\vec{c},(f_{\vec{\alpha}}))\in\hat{\F}_{[\vec{d}(\vec{\alpha})]} : \chi_{\ell(\vec{s})}(F_{(\vec{s})}^{(\vec{\omega})}(x_i)) = \epsilon_{\vec{s},\vec{\omega},i}, \vec{s}\in\s, \vec{\omega}\in\Omega_{\vec{s}}, i=1,\dots,q+1\}|} {|\hat{\F}_{[\vec{d}(\vec{\alpha})]}|}$$
$$ = \prod_{\substack{\vec{\beta}\in\R' \\ [\vec{\beta}]\not=[\vec{0}]}} \left( \frac{\phi(e(\vec{\beta})^2)} {|G|(q+|G| -1)} \right)^{m_{[\vec{\beta}]}}\left( \frac{q}{|G|(q+|G|-1)} \right)^{m_{[\vec{0}]}}\left(1 + O\left(q^{-\frac{\min(d(\vec{\alpha}))}{2}}\right)\right).$$

\end{cor}

\begin{proof}

Straight forward from Proposition \ref{valtakpropinf} and Corollary \ref{keypropcor}.
\end{proof}


\section{Proof of Theorem \ref{mainthm2}}\label{abelmainthm}

For any $(\vec{c},(f_{\vec{\alpha}}))\in\hat{\F}_{[\vec{d}(\vec{\alpha})]}$ and $x\in\Pp^1(\Ff_q)$,
$$\sum_{\vec{s}\in\s} \sum_{\vec{\omega}\in\Omega_{\vec{s}}} \chi_{\ell(\vec{s})}\left(F_{(\vec{s})}^{(\vec{\omega})}(x) \right) =  |\{\vec{s}\in\s, \vec{\omega}\in\Omega_{\vec{s}} : \chi_{\ell(\vec{s})}\left(F_{(\vec{s})}^{(\vec{\omega})}(x)\right) \not=0\}| $$
if $\chi_{\ell(\vec{s})}\left(F_{(\vec{s})}^{(\vec{\omega})}(x)\right) = 0$ or $1$ for all $\vec{s}\in\s, \vec{\omega}\in\Omega_{\vec{s}}$ and $0$ otherwise.

Now, if $\{\chi_{\ell(\vec{s})}\left(F_{(\vec{s})}^{(\vec{\omega})}(x) \right), \vec{s}\in\s, \vec{\omega}\in\Omega_{\vec{s}}\}$ is $[\vec{\beta}]$-admissible then
$$|\{\vec{s}\in\s, \vec{\omega}\in\Omega_{\vec{s}} : \chi_{\ell(\vec{s})}\left(F_{(\vec{s})}^{(\vec{\omega})}(x)\right) \not=0\}| = |A_{\vec{\beta}}| = \frac{|G|}{e(\vec{\beta})}.$$

Recall that $e(\vec{\beta}) = \lcm\left(\frac{r_j}{(r_j,\beta_j)}\right)|r_n$. Then the number of points lying over $x\in\Pp^1(\Ff_q)$ will be $\frac{|G|}{s_n}$ for some $s_n|r_n$.

\begin{prop}
Let $e_1,\dots,e_{q+1}$ be such that $e_i=0$ or $e_i = \frac{|G|}{s_{n,i}}$ for some $s_{n,i}|r_n$. For all $s|r_n$ let
$$m_s = |\{1\leq i \leq q+1: e_i=\frac{|G|}{s}\}|$$
and
$$m_0 = |\{1 \leq i \leq q+1 : e_i =0\}|$$
then
$$|\{(\vec{c},(f_{\vec{\alpha}})) \in \hat{\F}_{[\vec{d}(\vec{\alpha})]} : \sum_{\vec{s}\in\s} \sum_{\vec{\omega}\in\Omega_{\vec{s}}} \chi_{\ell(\vec{s})}\left(F_{(\vec{s})}^{(\vec{\omega})}(x_i)\right) =e_i, i=1,\dots,q+1 \}|$$
$$ = \left( \frac{(|G|-1)(q+|G|) -\sum_{s|r_n}s\phi_G(s)+1}{|G|(q+|G|-1)} \right)^{m_0} \left( \frac{q}{|G|(q+|G|-1)} \right)^{m_1} \prod_{\substack{s|r_n\\ s\not=1}} \left( \frac{s\phi_G(s)} {|G|(q+|G| -1)} \right)^{m_s} \times$$
$$ \left(1 + O\left(q^{-\frac{\min(d(\vec{\alpha}))}{2}}\right)\right)$$
where $\phi_G(s)$ is the number of elements of $G$ with order $s$.
\end{prop}

\begin{proof}

Let
$$M_s = \{1\leq i \leq q+1: e_i=\frac{|G|}{s}\}$$
$$M_0 = \{1 \leq i \leq q+1 : e_i =0\}.$$

If $i\in M_s$, $s\not=0$, then the set
$$\{\chi_{\ell(\vec{s})}(F_{(\vec{s})}^{(\vec{\omega})}(x_i)), \vec{s}\in\s, \vec{\omega}\in\Omega\}$$
will be $[\vec{\beta}]$-admissible for some $\vec{\beta}$ such that $e(\vec{\beta})=s$. Moreover, if $(\vec{s},\vec{\omega})\in A_{\vec{\beta}}$ then $\chi_{\ell(\vec{s})}(F_{(\vec{s})}^{(\vec{\omega})}(x_i))=1$.

Fix a partition of $M_s$ as
$$M_s = \bigcup_{\substack{[\vec{\beta}]\in\tilde{\R} \\ e(\vec{\beta})=s}} M_{[\vec{\beta}]} = \bigcup_{\substack{[\vec{\beta}]\in\tilde{\R} \\ e(\vec{\beta})=s}}\{ i\in M_s : \{\chi_{\ell(\vec{s})}(F_{(\vec{s})}^{(\vec{\omega})}(x_i)), \vec{s}\in\s, \vec{\omega}\in\Omega\} \mbox{ is $[\vec{\beta}]$-admissible} \}$$
and let $m_{[\vec{\beta}]} = |M_{[\vec{\beta}]}|$.

If $i\in M_0$, then the set
$$\{\chi_{\ell(\vec{s})}(F_{(\vec{s})}^{(\vec{\omega})}(x_i)), \vec{s}\in\s, \vec{\omega}\in\Omega\}$$
can be $[\vec{\beta}]$-admissible for any $[\vec{\beta}]\in\tilde{\R}$ as long as at least one of $\chi_{\ell(\vec{s})}(F_{(\vec{s})}^{(\vec{\omega})})\not=0$ or $1$.

Fix a partition of $M_0$ as
$$M_0 = \bigcup_{[\vec{\beta}]\in\tilde{\R}} M_{0,[\vec{\beta}]} = \bigcup_{[\vec{\beta}]\in\tilde{\R}}\{ i\in M_0 : \{\chi_{\ell(\vec{s})}(F_{(\vec{s})}^{(\vec{\omega})}(x_i)), \vec{s}\in\s, \vec{\omega}\in\Omega\} \mbox{ is $[\vec{\beta}]$-admissible} \}$$
and let $m_{0,[\vec{\beta}]} = |M_{0,[\vec{\beta}]}|$.

If $i\in M_{[\vec{\beta}]}$ then there is only one choice for the set $\{\chi_{\ell(\vec{s})}(F_{(\vec{s})}^{(\vec{\omega})}(x_i)), \vec{s}\in\s, \vec{\omega}\in\Omega\}$. (Namely, $\chi_{\ell(\vec{s})}(F_{(\vec{s})}^{(\vec{\omega})}(x_i))=1$ if $(\vec{s},\vec{\omega})\in A_{\vec{\beta}}$ and $0$ otherwise.) If $i\in M_{0,[\vec{\beta}]}$, then there will be $|A_{\vec{\beta}}|-1=\frac{|G|}{e(\vec{\beta})}-1$ choices for the set $\{\chi_{\ell(\vec{s})}(F_{(\vec{s})}^{(\vec{\omega})}(x_i)), \vec{s}\in\s, \vec{\omega}\in\Omega\}$.

Therefore,
$$\frac{|\{(\vec{c},(f_{\vec{\alpha}})) \in \hat{\F}_{[\vec{d}(\vec{\alpha})]} : \sum_{\vec{s}\in\s} \sum_{\vec{\omega}\in\Omega_{\vec{s}}} \chi_{\ell(\vec{s})}\left(F_{(\vec{s})}^{(\vec{\omega})}(x_i)\right) = \frac{|G|}{s_{n,i}}, i=1,\dots,q+1 \}|}{|\hat{\F}_{[\vec{d}(\vec{\alpha})]}|}$$
$$ = \sum_{M_{[\vec{\beta}]}}\sum_{M_{0,[\vec{\beta}]}} \sum_{\epsilon_{\vec{s},\vec{\omega},i}} \frac{|\{(\vec{c},(f_{\vec{\alpha}}))\in\hat{\F}_{[\vec{d}(\vec{\alpha})]} : \chi_{\ell(\vec{s})}(F_{(\vec{s})}^{(\vec{\omega})}(x_i)) = \epsilon_{\vec{s},\vec{\omega},i}, \vec{s}\in\s, \vec{\omega}\in\Omega_{\vec{s}}, i=1,\dots,q+1\}|}{{|\hat{\F}_{[\vec{d}(\vec{\alpha})]}}|} $$
where the first two sums are over all the partitions of $M_s$, $s|r_n$ and $M_0$, respectively, and the third sum is over all possible choices for $\{\chi_{\ell(\vec{s})}(F_{(\vec{s})}^{(\vec{\omega})}(x_i)), \vec{s}\in\s, \vec{\omega}\in\Omega\}$.

$$ =\sum_{M_{[\vec{\beta}]}}\sum_{M_{0,[\vec{\beta}]}} \sum_{\epsilon_{\vec{s},\vec{\omega},i}} \prod_{\substack{\vec{\beta}\in\R' \\ [\vec{\beta}]\not=[\vec{0}]}} \left( \frac{\phi(e(\vec{\beta})^2)} {|G|(q+|G| -1)} \right)^{m_{[\vec{\beta}]}+m_{0,[\vec{\beta}]}}\left( \frac{q}{|G|(q+|G|-1)} \right)^{m_{[\vec{0}]} + m_{0,[\vec{0}]}} \times$$
$$ \left(1 + O\left(q^{-\frac{\min(d(\vec{\alpha}))}{2}}\right)\right)$$
$$ =\sum_{M_{[\vec{\beta}]}} \prod_{\substack{s|r_n\\ s\not=1}} \left( \frac{\phi(s^2)} {|G|(q+|G| -1)} \right)^{m_s}\left( \frac{q}{|G|(q+|G|-1)} \right)^{m_1} \times$$
$$ \sum_{M_{0,[\vec{\beta}]}}\prod_{\substack{\vec{\beta}\in\R' \\ [\vec{\beta}]\not=[\vec{0}]}} \left( \frac{\phi(e(\vec{\beta}))(|G|-e(\vec{\beta}))} {|G|(q+|G| -1)} \right)^{m_{0,[\vec{\beta}]}}\left( \frac{(|G|-1)q}{|G|(q+|G|-1)} \right)^{m_{0,[\vec{0}]}} \left(1 + O\left(q^{-\frac{\min(d(\vec{\alpha}))}{2}}\right)\right)$$
$$ = \prod_{\substack{s|r_n\\ s\not=1}} \left( \frac{\phi(s^2)\sum_{e([\vec{\beta}])=s} 1} {|G|(q+|G| -1)} \right)^{m_s}\left( \frac{(|G|-1)q + \sum_{\substack{[\vec{\beta}]\in\tilde{\R} \\ [\vec{\beta}]\not=[\vec{0}]}} \phi(e(\vec{\beta}))(|G|-e(\vec{\beta}))}{|G|(q+|G|-1)} \right)^{m_0}\times$$
$$\left(\frac{q}{|G|(q+|G|-1)} \right)^{m_1} \left(1 + O\left(q^{-\frac{\min(d(\vec{\alpha}))}{2}}\right)\right).$$

First note that since there exists $\phi(e(\vec{\beta}))$ such $\vec{\beta'}$ such that $[\vec{\beta}']=[\vec{\beta}]$ so we can write
$$\phi(s^2)\sum_{e([\vec{\beta}])=s} 1 = s \sum_{e(\vec{\beta})=s} 1$$
and
$$\sum_{\substack{[\vec{\beta}]\in\tilde{\R} \\ [\vec{\beta}]\not=[\vec{0}]}} \phi(e(\vec{\beta}))(|G|-e(\vec{\beta})) = \sum_{\vec{\beta}\in\R} (|G|-e(\vec{\beta})) = (|G|-1)|G| - \sum_{\vec{\beta}\in\R} e(\vec{\beta}).$$

Now, for every $\vec{\beta}\in\R'$, we can view it in a natural way as element of $G$. Moreover, the order of $\vec{\beta}$ would be $e(\vec{\beta})$. Hence $s \sum_{e(\vec{\beta})=s} 1 = s\phi_G(s)$. Further
$$\sum_{\vec{\beta}\in\R} e(\vec{\beta}) = \sum_{\substack{s|r_n \\ s\not=1}} s \sum_{e(\vec{\beta})=s} 1 = \sum_{s|r_n} s\phi_G(s)-1.$$

\end{proof}

%

Finally, we end it with the proof of Theorem \ref{mainthm2}.

\begin{proof}[Proof of Theorem \ref{mainthm2}]

$$\frac{|\{C\in\Hh^{(\vec{d}(\vec{\alpha}))} : \#C(\Pp^1(\Ff_q)) = M\}|}{|\Hh^{(\vec{d}(\vec{\alpha}))}|}$$
$$ = \sum_{\substack{e_1,\dots,e_{q+1} \\ \sum e_i = M}} |\{(\vec{c},(f_{\vec{\alpha}})) \in \hat{\F}_{[\vec{d}(\vec{\alpha})]} : \sum_{\vec{s}\in\s} \sum_{\vec{\omega}\in\Omega_{\vec{s}}} \chi_{\ell(\vec{s})}\left(F_{(\vec{s})}^{(\vec{\omega})}(x_i)\right) =e_i, i=1,\dots,q+1 \}|$$
$$ = \sum_{\substack{e_1,\dots,e_{q+1} \\ \sum e_i = M}} \left( \frac{(|G|-1)(q+|G|) -\sum_{s|r_n}s\phi_G(s)+1}{|G|(q+|G|-1)} \right)^{m_0} \left( \frac{q}{|G|(q+|G|-1)} \right)^{m_1} \times $$
$$\prod_{\substack{s|r_n\\ s\not=1}} \left( \frac{s\phi_G(s)} {|G|(q+|G| -1)} \right)^{m_s} \left(1 + O\left(q^{-\frac{\min(d(\vec{\alpha}))}{2}}\right)\right)$$
$$ = \Prob\left(\sum_{i=1}^{q+1} X_i= M\right)\left(1 + O\left(q^{-\frac{\min(d(\vec{\alpha}))}{2}}\right)\right).$$
\end{proof}

\textbf{Acknowledgements:} I would like to thank Chantal David for the countless discussions we had about this topic. I would also like to thank Elisa Lorenzo, Giulio Meleleo and Piermarco Milione for their helpful discussions about their paper which shed light on how to approach this problem.

\bibliography{AbelianCovers}

\providecommand{\bysame}{\leavevmode\hbox to3em{\hrulefill}\thinspace}
\providecommand{\MR}{\relax\ifhmode\unskip\space\fi MR }
\providecommand{\MRhref}[2]{%
  \href{http://www.ams.org/mathscinet-getitem?mr=#1}{#2}
}
\providecommand{\href}[2]{#2}
\begin{thebibliography}{10}

\bibitem{BDFK+}
Alina Bucur, Chantal David, Brooke Feigon, Nathan Kaplan, Matilde Lal{\i}n,
  Ekin Ozman, and Melanie~Mathett Wood, \emph{The distribution of points on
  cyclic covers of genus g}, preprint (2015).

\bibitem{BDFL1}
Alina Bucur, Chantal David, Brooke Feigon, and Matilde Lal{\i}n, \emph{Biased
  statistics for traces of cyclic p-fold covers over finite fields}, WIN--Women
  in Numbers: Research Directions in Number Theory \textbf{60} (2009),
  121--143.

\bibitem{BDFL2}
Alina Bucur, Chantal David, Brooke Feigon, and Matilde Lal{\'\i}n,
  \emph{Statistics for traces of cyclic trigonal curves over finite fields},
  International Mathematics Research Notices (2009), rnp162.

\bibitem{DF}
David~Steven Dummit and Richard~M Foote, \emph{Abstract algebra}, vol. 1984,
  Wiley Hoboken, 2004.

\bibitem{hart}
Robin Hartshorne, \emph{Algebraic geometry}, vol.~52, Springer Science \&
  Business Media, 1977.

\bibitem{KS}
Nicholas~M Katz and Peter Sarnak, \emph{Random matrices, frobenius eigenvalues,
  and monodromy}, vol.~45, American Mathematical Soc., 1999.

\bibitem{KR}
P{\"a}r Kurlberg and Ze{\'e}v Rudnick, \emph{The fluctuations in the number of
  points on a hyperelliptic curve over a finite field}, Journal of Number
  Theory \textbf{129} (2009), no.~3, 580--587.

\bibitem{LMM}
Elisa Lorenzo, Giulio Meleleo, Piermarco Milione, and Alina Bucur,
  \emph{Statistics for biquadratic covers of the projective line over finite
  fields}, to appear in Journal of Number Theory (2015).

\bibitem{M1}
Patrick Meisner, \emph{Distribution of points on cyclic curves over finite
  fields}, arXiv preprint arXiv:1511.07814 (submitted) (2015).

\bibitem{rose}
Michael Rosen, \emph{Number theory in function fields}, vol. 210, Springer
  Science \& Business Media, 2013.

\end{thebibliography}
\bibliographystyle{amsplain}

\end{document}